\documentclass[12pt]{amsart}

\usepackage{amscd,amsfonts,amssymb,amsmath}

\copyrightinfo{2000}{American Mathematical Society}

\numberwithin{equation}{section}

\newtheorem{thm}{Theorem}[section]
\newtheorem{lem}[thm]{Lemma}

\newtheorem{co}[thm]{Conjecture}
\newtheorem{cor}[thm]{Corollary}
\newtheorem{rem}[thm]{Remark}
\newtheorem{df}[thm]{Definition}
\newtheorem{ex}[thm]{Example}

\author{A.~V.~Kosyak}
\address{Institute of Mathematics, Ukrainian National Academy of Sciences, 3
Tereshchenkivs'ka, Kyiv, 01601, Ukraine.}
\email{kosyak02@gmail.com}

\title{Induced representations of infinite-dimensional groups}

\date{}

\begin{document}

\begin{abstract}
The {\it induced representation} ${\rm Ind}_H^GS$ of a locally
compact group $G$ is the unitary representation of the group $G$
associated with unitary representation $S:H\rightarrow U(V)$ of a
subgroup $H$ of the group $G$. Our aim is to {\it develop the
concept of induced representations for infinite-dimensional groups}.
The induced representations for infinite-dimensional groups in not
unique, as in the case of a locally compact groups. It depends on
two  completions $\tilde H$ and $\tilde G$ of the subgroup $H$ and
the group $G$, on an extension $\tilde S:\tilde H\rightarrow U(V)$
of the representation $S:H\rightarrow U(V)$ and on a choice of the
$G$-quasi-invariant measure $\mu$ on an appropriate completion
$\tilde X=\tilde H\backslash \tilde G$  of the space $H\backslash
G$. As the illustration we consider the ``nilpotent'' group
$B_0^{\mathbb Z}$ of infinite in both directions upper triangular matrices and the
induced representation corresponding to the so-called {\it generic
orbits}.
\end{abstract}
\maketitle
\tableofcontents
\section{Introduction}
The {\it induced representations} were introduced and studied for a {\it finite groups} by F.G. Frobenius.
Our aim is to  develop the concept of {\it induced representations
for infinite-dimensional groups}.

The content of the article is as
follows. Section~\ref{s.ind-loc-comp} is devoted to  the notion of
induced representations elaborated for a {\it locally compact groups} by
G.W.Mackey \cite{Mack49,Mack52} and to the Kirillov {\it orbit methods}
\cite{Kir62} for the nilpotent Lie groups $B(n,{\mathbb R})$.

In Section~\ref{s.ind-inf} we extend the notion of the induced
representations for infinite-dimensi\-onal groups. We start the
orbit method for infinite-dimensional ``nilpotent'' group
$B_0^{\mathbb Z}$, construct the induced representations
corresponding to the generic orbits and study its irreducibility.

In Section~\ref{Append-B-0^Z} we remind the Gauss decomposition  of $n\times n$ matrices (Subsection~\ref{Gauss-dec}),
and Gauss  decomposition of infinite order matrices (Subsection~\ref{Gauss-dec-inf}).

More precisely, we give  the well-known definition of the induced representations for a locally compact
groups in Subsection~\ref{Ind-loc}.
In Subsection~\ref{orb-meth(n)} we remind the Kirillov orbit method for finite-dimensional nilpotent group $G_n=B(n,{\mathbb R})$.
The induced  representations, corresponding to a generic orbits of the group $G_n$ are discussed in Subsection~\ref{orb-gen-loc}.
In the Subsection~\ref{new-irred-loc} we give a new proof of the irreducibility of the induced representations corresponding to
a generic orbits in order to extend  the proof of the irreducibility for  infinite-dimensional ``nilpotent''  group $B_0^{\mathbb Z}$.

In Subsection~\ref{reg+quasireg-rep} we remind the definition of
the regular and quasiregular representations of
infinite-dimensional groups. As in the case of a locally compact
group these representations are the particular cases of the
induced representations. This gives us the hint how to define the
induced representations for infinite-dimensional groups. The
definition is done in Subsection~\ref{Ind-rep}. The questions
concerning the  development of the orbit method for
infinite-dimensional ``nilpotent'' group $B_0^{\mathbb N}$ and
$B_0^{\mathbb Z}$ are discussed in Subsection~\ref{orb-meth-inf}.

The completions of the initial groups $G$ are necessary to the
definition of the induced representations for the initial
infinite-dimensional group. The completions of the inductive limit
$G=\varinjlim_{n}G_n$ of matrix groups $G_n$ are studied in
Subsection~\ref{Hilb-Lie-GL_2(a)} and \ref{Hilb-Lie-B_2(a)}. We
show that the {\it Hilbert-Lie groups} appear naturally in the
representation theory of the infinite-dimensional matrix group.
We define a family of the Hilbert-Lie group ${\rm GL}_2(a)$ (resp.
$B_2(a)$), a Hilbert completions of the group ${\rm
GL}_0(2\infty,{\mathbb R})=\varinjlim_{n}{\rm GL}(2n-1,{\mathbb R})$ (resp. $B_0^{\mathbb
Z}=\varinjlim_{n}B(2n-1,{\mathbb R})$). We show that {\it any  continuous representation} of the group
${\rm GL}_0(2\infty,{\mathbb R})$ (resp. $B_0^{\mathbb
Z}$) {\it is in fact continuous} in some stronger topology, namely
{\it in a topology of a suitable Hilbert -Lie group}
${\rm GL}_2(a)$ (resp. $B_2(a)$) depending on the representation.

In Subsection~\ref{gen-orb,Z} we construct the induced
representations of the group $B_0^{\mathbb Z}$ corresponding to a
generic orbits. The irreducibility of these representations is
studied in Subsection~\ref{gen-orb,Z,irr}. The very first steps to
describe some part of the {\it dual} for the group $B_0^{\mathbb N}$ and
$B_0^{\mathbb Z}$ are mentioned in Subsection~\ref{dual-astep}
\section{Induced  representations, finite-dimensional case}
\label{s.ind-loc-comp}
\subsection{Induced representations}
\label{Ind-loc}

\index{representation!induced}
\index{section!Borel}
The {\it induced representation} ${\rm Ind}_H^GS$  is the unitary representation of a group $G$
associated with a unitary representation $S:H\rightarrow U(V)$ of a closed subgroup $H$ of
the group $G$. For details, see \cite{Kir94}, Section 2.1. Suppose
that $X=H\setminus G$ is a right $G-$space and that $s:X\rightarrow G$ is a Borel
section of the projection $p:G\rightarrow X=H\backslash
G\,:\,g\mapsto Hg$. For Lie group, such a mapping $s$ can be
chosen to be smooth almost everywhere. Then every element $g\in G$
can be uniquely written in the form
\begin{equation}
\label{9.kir.6.3} g=h s(x),\,h\in H,\,\,x\in X,
\end{equation}
and thus $G$ (as a set) can be identified with $H\times X$. Under
this identification, the Haar measure on $G$ goes into a measure
equivalent to the product of a quasi-invariant measure on $X$ and a
Haar measure on $H$. More precisely, if a quasi-invariant measure
$\mu_s$ on $X$ is appropriately chosen, then the following
equalities are valid
\begin{equation}
\label{9.kir.6.4}
d_r(g)=\frac{\Delta_G(h)}{\Delta_H(h)}d\mu_s(x)d_r(h),
\end{equation}
\begin{equation}
\label{9.kir.6.5}
\frac{d\mu_s(xg)}{d\mu_s(x)}=\frac{\Delta_H(h(x,g))}{\Delta_G(h(x,g))},
\end{equation}
where $\Delta_G$ is a modular function on the group $G$ and $h(x,g)\in H$ is defined by the relation
\begin{equation}
\label{9.kir.6.6} s(x)g=h(x,g)s(xg).
\end{equation}
Recall that a {\it modular  function} on a group $G$ is a
homomorphism $G\ni t\mapsto \Delta_G(t)\in {\mathbb R}_+$ defined
by the equality $h^{L_t}=\Delta_G(t)h$, where $h$ is the right
Haar measure on $G$, $L$ is the left action  of the group $G$ on itself and
$h^{L_t}(C)=h(tC)$.
\index{group!unimodular}
\begin{rem}
\label{9.r.s:X-G}
 If the group $G$ is {\it unimodular}, i.e
$\Delta_G\equiv 1$, and it is possible to select a subgroup $K$
that is complementary to $H$ in the sense that almost every
element of $G$ can be uniquely written in the form
\begin{equation}
\label{9.kir.6.7} g=h k,\,\,h\in H,\,\,k\in K,
\end{equation}
then it is natural to identify $X=H\backslash G$ with $K$ and to
choose $s$ as the embedding of $K$ in $G$
\begin{equation}
\label{9.df.s:X-G} s:K\mapsto G.
\end{equation}
In such a case, the formula (\ref{9.kir.6.4}) assume the form
\begin{equation}
\label{9.kir.6.4'} dg=\Delta_H(h)^{-1}d_r(h)d_r(k).
\end{equation}
\end{rem}
If both $G$ and $H$ are unimodular (or, more generally, if
$\Delta_G(h)$ and $\Delta_H(h)$ coincide for $h\in H$), then there
exist a $G$-invariant measure on $X\!\!=\!\!H\backslash G$. If it is
possible to extend $\Delta_H$ to a multiplicative function on the
group $G$, then there exist a {\it quasi-in\-variant measure} on $X$
which is multiplied by the factor
$\!\frac{\Delta_H(g)}{\Delta_G(g)}\!$ under translation by $g$.

Now we can define ${\rm Ind}_H^GS$ (see \cite{Kir94}, section
2.3.). Let $S:H\rightarrow U(V)$ be a unitary representation of a
subgroup $H$ of the group $G$ in a Hilbert space $V$ and let $\mu$
be a measure on $X$ satisfying condition (\ref{9.kir.6.5}). Let
${\mathcal H}$ denote the space of all vector-valued functions $f$
on $X$ with values in $V$ such that
$$
\Vert f\Vert^2:=\int_X \Vert f(x)\Vert^2_Vd\mu(x)<\infty.
$$
Let us consider the representation $T$ given by the formula
\begin{equation}
 [T(g)f](x)=A(x,g)f(xg)=S(h)\left(\frac{d\mu_s(xg)}{d\mu_s(x)}\right)^{1/2}f(xg),
\end{equation}
where
\begin{equation}
\label{A(x,g)} A(x,g)=\left[\frac{\Delta_H(h)}{\Delta_G(h)}
\right]^{1/2}S(h),
\end{equation}
and where the element $h=h(x,g)$ is defined by formula
(\ref{9.kir.6.6}).
\begin{df}\label{9.df.ind-rep}
The representation $T$ is called the {\rm unitary induced}
representation and is denoted by ${\rm Ind}_H^GS$.
\end{df}
\index{representation!unitary induced}
\begin{rem}
\label{reg=ind} The {\rm right (or the left) regular
representation} $\rho,\lambda:G\mapsto U(L^2(G,h))$ of a locally
compact group $G$ is a particular case of the induced
representation ${\rm Ind}^G_HS$ with $H=\{e\}$ and $S=Id$. The
{\rm guasiregular representation}
is a particular
case of the induced representation
with some closed subgroup $H\subset G$ and $S=Id$.
\end{rem}

\subsection{Orbit method for finite-dimensional nilpotent group $B(n,{\mathbb R})$}
\label{orb-meth(n)}
\index{orbit method}
See Kirillov~\cite{Kir76} and \cite{Kir94}, Chapter 7, \S 2, p.129-130, for
details. "Fix the group $G_n=B(n,{\mathbb R})$ of all upper
triangular real matrices of order $n$ with ones on the main
diagonal. (The Kirillov notation for the group $B(n,{\mathbb R})$ is
$N_+(n,{\mathbb R})$).

The basic result of the method of orbits, applied to nilpotent Lie
groups, is the description of a one-to-one correspondence between
two sets:
\index{group!Lie!nilpotent}

a) the set $\hat G$ of all equivalence classes of irreducible
unitary representations of a connected and simply connected
nilpotent Lie group $G$,
\index{group!connected}

b) the set ${\mathcal O}(G)$ of all orbits of the group $G$ in the
space ${\mathfrak g}^*$ dual to the Lie algebra ${\mathfrak g}$
with respect to the coadjoint representation.
\index{orbit!set}

To construct this correspondence, we introduce the following
definition. A subalgebra ${\mathfrak h}\subset {\mathfrak g}$ is
{\it subordinate} to a functional $f\in {\mathfrak g}^*$ if
\index{subalgebra!subordinate}
$$
\langle f,[x,y]\rangle=0\quad{for\,\,\, all}\quad x,y\in {\mathfrak
h},
$$
i.e. if ${\mathfrak h}$ is an {\it isotropic subspace} with respect to the
bilinear form  defined by $B_f(x,y)=\langle f,[x,y]\rangle$ on
${\mathfrak g}$.
\index{subspace!isotropic}
\begin{lem}[Lemma 7.7, \cite{Kir94}]
\label{9.l.Kir} The following conditions are equivalent:
\par (a) a subalgebra ${\mathfrak h}$ is subordinate to the
functional $f$,
\par (b) the image of ${\mathfrak h}$ in the tangent space
$T_f\Omega$ to the orbit $\Omega$ in the point $f$ is an isotropic
subspace,
\index{space!tangent}
\par (c) the map
$$
x\mapsto\langle f,x\rangle
$$
is a one-dimensional real representation of the Lie algebra
${\mathfrak h}$.
\end{lem}
If the conditions of Lemma \ref{9.l.Kir} are satisfied, we define
the one-dimensional unitary representation $U_{f,H}$ of the group
$H=\exp {\mathfrak h}$ by the formula
$$
U_{f,H}(\exp\, x)=\exp\, 2\pi i\langle f,x\rangle.
$$
\begin{thm}[Theorem 7.2, \cite{Kir94}]
\label{9.t.Kir}
\par (a) Every irreducible unitary representation $T$ of a
connected and simply connected nilpotent Lie group $G$ has the form
 $$
T={\rm Ind}^G_H U_{f,H},
 $$
where $H\subset G$ is a connected subgroup and $f\in{\mathfrak
g}^*$;
\index{subgroup!connected}
\par (b) the representation $T_{f,H}={\rm Ind}^G_H U_{f,H}$ is
irreducible if and only if the Lie algebra ${\mathfrak h}$ of the
group $H$ is a subalgebra of ${\mathfrak g}$ subordinate to the
functional $f$ with maximal possible dimension;
\par (c) irreducible representations $T_{f_1,H_1}$ and
$T_{f_2,H_2}$ are equivalent if and only if  the functionals $f_1$
and $f_2$ belong to the same orbit of ${\mathfrak g}^*$."
 \end{thm}

\begin{ex}
\label{Geiseberg-gr} Let us consider the Heisenberg group $G_3=B(3,{\mathbb R})$, its Lie algebra ${\mathfrak g}$ and the dual space
${\mathfrak g}^*$.  Fix the notations
\end{ex}
$$
G=B(3,{\mathbb R})=
 \left\{
 \left(\begin{smallmatrix}
1&x_{12}&x_{13}\\
0&1&x_{23}\\
0&0&1\\
\end{smallmatrix}\right)
\right\},
$$
$$
\,\, {\mathfrak g}=n_+(3,{\mathbb R})=
\left\{\left(\begin{smallmatrix}
0&x_{12}&x_{13}\\
0&0&x_{23}\\
0&0&0\\
\end{smallmatrix}\right)\right\}
,\,\,{\mathfrak g}^*= n_-(3,{\mathbb R})=
\left\{\left(\begin{smallmatrix}
0&0&0\\
y_{21}&0&0\\
y_{31}&y_{32}&0\\
\end{smallmatrix}\right)\right\}.
$$
The {\it adjoint action} ${\rm Ad}:G\rightarrow {\rm Aut}(\mathfrak g)$ of the group $G$ on
its Lie algebra ${\mathfrak g}$ is:
\begin{equation}
\label{Ad(3)} {\mathfrak g}\ni x \mapsto {\rm Ad}_t(x):=txt^{-1}\in
{\mathfrak g},\quad t\in G,
\end{equation}
the pairing between   the ${\mathfrak g}$ and ${\mathfrak g}^*$:
\begin{equation}
\label{<,>3} {\mathfrak g}^*\times {\mathfrak g}\ni (y,x)\mapsto
\langle y,x\rangle :=tr(xy)=\sum_{1\leq k<n\leq
3}x_{kn}y_{nk}\in{\mathbb R}.
\end{equation}
Since $tr(txt^{-1}y)=tr(xt^{-1}yt)$ the {\it coadjoint action} of $G$
on the dual ${\mathfrak g}^*$ to ${\mathfrak g}$ is
\begin{equation}
 \label{Ad*(3)}
 {\mathfrak g}^*\ni y \mapsto {\rm Ad}^*_t(y):=(t^{-1}yt)_-\in
{\mathfrak g}^*,\quad t\in G,
\end{equation}
where $(z)_-$ means that we take  lower triangular part of the
matrix $z$.
\index{action!adjoint}
\index{action!coadjoint}

To calculate ${\rm Ad}^*_t(y)$ explicitly for $n=3$, we have
$$
t^{-1}yt=\left(\begin{smallmatrix}
1&t_{12}&t_{13}\\
0&1&t_{23}\\
0&0&1\\
\end{smallmatrix}\right)^{-1}
\left(\begin{smallmatrix}
0&0&0\\
y_{21}&0&0\\
y_{31}&y_{32}&0\\
\end{smallmatrix}\right)
\left(\begin{smallmatrix}
1&t_{12}&t_{13}\\
0&1&t_{23}\\
0&0&1\\
\end{smallmatrix}\right)
$$
$$
 =\left(\begin{smallmatrix}
1&-t_{12}&-t_{13}+t_{12}t_{23}\\
0&1&-t_{23}\\
0&0&1\\
\end{smallmatrix}\right)
\left(\begin{smallmatrix}
0&0&0\\
y_{21}&y_{21}t_{12}&y_{21}t_{13}\\
 y_{31}&  y_{31}t_{12}+y_{32}&y_{31}t_{13}+y_{32}t_{23}\\
\end{smallmatrix}\right),
$$
hence
$$ {\rm Ad}^*_t(y):=(t^{-1}yt)_-=
\left(\begin{smallmatrix}
0                   &0&0\\
y_{21}- t_{23}y_{31}&0&0\\
y_{31}              &y_{31}t_{12}+y_{32}&0\\
\end{smallmatrix}\right).
$$
We have two type of the orbits ${\mathcal O}$:
\par 1) if $y_{31}=0$, then
$\left(\begin{smallmatrix} y_{21}&\\
0&y_{32}\\
\end{smallmatrix}\right)\!\simeq\!(y_{21},y_{32})$ for fixed $y_{21},\,y_{32}$ is
$0$-dimensional orbit;
\par 2) if $y_{31}\not=0$, then $\left(\begin{smallmatrix} {\mathbb R}&\\
y_{31}&{\mathbb R}
\end{smallmatrix}\right)$ is $2$-dimensional orbits.

In the case 1) fixe the point $f=(y_{21},y_{32})$, the subordinate
subalgebra ${\mathfrak h}$ coinside with all ${\mathfrak g}$,
since $[{\mathfrak g},{\mathfrak g}]=\langle E_{13}\rangle:=\{tE_{13}\mid t\in{\mathbb R}\}$.
Corresponding one-dimensional representation of the algebra
${\mathfrak h}={\mathfrak g}$ is
$$
{\mathfrak g}\ni x\mapsto \langle f,x\rangle=
tr(xf)=
tr\left[ \left(\begin{smallmatrix}
0&x_{12}&x_{13}\\
0&0&x_{23}\\
0&0&0\\
\end{smallmatrix}\right)\left(\begin{smallmatrix}
0&0&0\\
y_{21}&0&0\\
0&y_{32}&0\\
\end{smallmatrix}\right)
\right] =x_{12}y_{21}+x_{23}y_{32}\in {\mathbb R}.
$$
The corresponding representation of the group $G$ is
\begin{equation}
\label{Orb(3),0} G\ni \exp(x)\mapsto \exp (2\pi i\langle
f,x\rangle)\in S^{1}.
\end{equation}
So we have 1-dimensional representation
$$
G_3\ni \exp\left(\begin{smallmatrix}
0&x_{12}&x_{13}\\
0&0&x_{23}\\
0&0&0\\
\end{smallmatrix}\right)\mapsto \exp (2\pi i(x_{12}y_{21}+x_{23}y_{32}))\in
S^{1}.
$$
We note that
$$
\exp(x)=\exp\left(\begin{smallmatrix}
0&x_{12}&x_{13}\\
0&0&x_{23}\\
0&0&0\\
\end{smallmatrix}\right)=\left(\begin{smallmatrix}
1&x_{12}&x_{13}+\frac{1}{2}x_{12}x_{23}\\
0&1&x_{23}\\
0&0&1\\
\end{smallmatrix}\right).
$$
In the case 2) we have two subordinate subalgebras of  the maximal dimension
$$
{\mathfrak h}_1= \left(\begin{smallmatrix}
0&0&x_{13}\\
0&0&x_{23}\\
0&0&0\\
\end{smallmatrix}\right),\quad\text{and}\quad
{\mathfrak h}_2= \left(\begin{smallmatrix}
0&x_{12}&x_{13}\\
0&0&0\\
0&0&0\\
\end{smallmatrix}\right).\quad\text{ Set}\quad f=\left(\begin{smallmatrix}
0&0&0\\
y_{21}&0&0\\
y_{31}&y_{32}&0\\
\end{smallmatrix}\right).
$$
The corresponding one-dimensional representations of the subalgebras ${\mathfrak h}_i,\,i=1,2$
are
$$
{\mathfrak h}_1\ni x\mapsto  \langle
f,x\rangle=x_{13}y_{31}+x_{23}y_{32}\in {\mathbb R},
$$
$$
{\mathfrak h}_2\ni x\mapsto \langle
f,x\rangle=x_{12}y_{21}+x_{13}y_{31}\in {\mathbb R}.
$$
 The corresponding representations $S$ of the subgroups
$H_1$ and $H_2$ respectively are:
$$
H_1\ni\left(\begin{smallmatrix}
1&0&x_{13}\\
0&1&x_{23}\\
0&0&1\\
\end{smallmatrix}\right)= \exp(x)\mapsto \exp(2\pi
i(x_{13}y_{31}+x_{23}y_{32}))\in S^1,
$$
$$
H_2\ni\left(\begin{smallmatrix}
1&x_{12}&x_{13}\\
0&1&0\\
0&0&1\\
\end{smallmatrix}\right)= \exp(x)\mapsto \exp(2\pi
i(x_{12}y_{21}+x_{13}y_{31}))\in S^1.
$$
In the case $H_1$ we have the decomposition $G_3={\mathbb R}^2
\ltimes B(2,{\mathbb R})\simeq H_1\ltimes {\mathbb R}$, indeed we
have
$$
G_3\ni\left(\begin{smallmatrix}
1&x_{12}&x_{13}\\
0&1&x_{23}\\
0&0&1\\
\end{smallmatrix}\right)=\left(\begin{smallmatrix}
1&0&x_{13}\\
0&1&x_{23}\\
0&0&1\\
\end{smallmatrix}\right)\left(\begin{smallmatrix}
1&x_{12}&0\\
0&1&0\\
0&0&1\\
\end{smallmatrix}\right)\in {\mathbb R}^2\ltimes B(2,{\mathbb R}),
$$
hence the space $X=H_1\backslash G_3$ is isomorphic to $B(2,{\mathbb R})\simeq{\mathbb R}$ and $s$ can be choosing as  the {\it embedding}
$s:B(2,{\mathbb R})\mapsto B(3,{\mathbb R})$.
$$
B(2,{\mathbb R})\ni \left(\begin{smallmatrix}
1&x\\
0&1\\
\end{smallmatrix}\right)=:x\mapsto s(x)=\left(\begin{smallmatrix}
1&x&0\\
0&1&0\\
0&0&1\\
\end{smallmatrix}\right)\in B(3,{\mathbb R}).
$$
For general $n$ we have
\begin{equation}
B(n+1,{\mathbb R})={\mathbb R}^n \ltimes B(n,{\mathbb R}).
\end{equation}

To calculate the right action of $G$ on $X$ i.e. to find $h(x,t)$
such that
$$
s(x)t=h(x,t)s(xt),
$$
we have for $x\in B(2,{\mathbb R})$ and $t\in B(3,{\mathbb R})$
$$
s(x)t=\left(\begin{smallmatrix}
1&x&0\\
0&1&0\\
0&0&1\\
\end{smallmatrix}\right)\left(\begin{smallmatrix}
1&t_{12}&t_{13}\\
0&1&t_{23}\\
0&0&1\\
\end{smallmatrix}\right)=\left(\begin{smallmatrix}
1&x+t_{12}&t_{13}+xt_{23}\\
0&1&t_{23}\\
0&0&1\\
\end{smallmatrix}\right)
=\left(\begin{smallmatrix}
1&0&t_{13}+xt_{23}\\
0&1&t_{23}\\
0&0&1\\
\end{smallmatrix}\right)
\left(\begin{smallmatrix}
1&x+t_{12}&0\\
0&1&0\\
0&0&1\\
\end{smallmatrix}\right)
$$
$$=h(x,t)s(xt),
\text{\,\,hence\,\, }h(x,t)=\left(\begin{smallmatrix}
1&0&t_{13}+xt_{23}\\
0&1&t_{23}\\
0&0&1\\
\end{smallmatrix}\right).
$$
 Finally, the induced unitary representation ${\rm Ind}_{H_1}^GS$
have the following form in the Hilbert  space $L^2({\mathbb R}, dx)$
(case $H_1$ and $f=y_{31}E_{31}$):
\begin{equation}
\label{Orb(3),2} f(x)\mapsto S(h(x,t))f(xt)=\exp (2\pi
i(t_{13}+t_{23}x)y_{31})f(x+t_{12}).
\end{equation}
In the Kirillov \cite{Kir94} notations we have:
$$
f(x)\mapsto \exp (2\pi i(c+bx)\lambda)f(x+a),\quad
y_{31}=\lambda,\,\, \left(\begin{smallmatrix}
1&t_{12}&t_{13}\\
0&1&t_{23}\\
0&0&1\\
\end{smallmatrix}\right)=\left(\begin{smallmatrix}
1&a&c\\
0&1&b\\
0&0&1\\
\end{smallmatrix}\right).
$$
%
\subsection{The induced  representations, corresponding to a generic orbits, finite-dimensional case}
\label{orb-gen-loc}
We show following A.~Kirillov \cite{Kir94} how the orbit method works for the nilpotent group
$B(n,{\mathbb R})$ and small $n$.

For general $n\in{\mathbb N}$ the coadjoint action of the group $G_n$ on ${\mathfrak g}$ is as follows
$$
t=I\!+\!\!\sum_{1\leq k<m\leq n}t_{km}E_{km},\,\, y=\sum_{1\leq
m<k\leq n}y_{km}E_{km},\,\,t^{-1}:=I\!+\!\!\sum_{1\leq k<m\leq
n}t_{km}^{-1}E_{km}
$$
hence
$$
(tyt^{-1})_{pq}=\sum_{m=1}^q(ty)_{pm}t^{-1}_{mq}=\sum_{m=1}^q\sum_{r=p}^nt_{pr}y_{rm}t^{-1}_{mq},\,\,
1\leq p,q\leq n,
$$
and
\begin{equation} \label{9.Ad(n)} {\rm
Ad}^*_t(y)=(t^{-1}yt)_-=I+\sum_{1\leq q<p\leq
n}(t^{-1}yt)_{pq}E_{pq}.
\end{equation}
\begin{ex}
\label{Gen-orbit.n} {\rm Generic orbits} for the group
$G=B(n,{\mathbb R})$ {\rm (}see {\rm \cite{Kir94}}, Example 7.9{\rm
)}.
\end{ex}
``The form of the action ${\rm Ad}^*_t(y)=(t^{-1}yt)_-$ implies,
that ${\rm Ad}^*_t,\,t\in G$ acts as follows: to a given column of
$y\in {\mathfrak g}^*$, a linear combination of the previous columns
is added and to a given row of $y$,  a linear combination of the
following rows is added. More generally, the minors
$\Delta_k,\,k=1,2,...,[\frac{n}{2}]$, consisting of the last $k$
rows and first $k$ columns of $y$ are invariant of the action. {\it
It is possible to show that if all the numbers $c_k$ are different
from zeros, then the manifold given by the equation
\begin{equation}
\label{Gen-orb}
\Delta_k=c_k,\,\,1\leq k\leq  \left[\frac{n}{2}\right]
\end{equation}
is a $G$-orbit in} ${\mathfrak g}^*$. Hence generic orbits have
codimension equal to $[\frac{n}{2}]$ and dimension equal to
$\frac{n(n-1)}{2}-[\frac{n}{2}]$. To obtain a representation for
such an orbit, we can take a matrix $y$ of the form
$$
y=\left(\begin{smallmatrix} 0&0\\
\Lambda&0\\
\end{smallmatrix}\right),
$$
where $\Lambda$ is the matrix of order $[\frac{n}{2}]$ such that all
nonzero elements are contained in the {\it anti-diagonal}. It is easy to
find a subalgebra of dimension $[\frac{n}{2}]\times[\frac{n+1}{2}]$
subordinate to the functional $y$. It consist of all matrices of
the form
$$
\left(\begin{smallmatrix} 0&A\\
0&0\\
\end{smallmatrix}\right),
$$
where $A$ is an $[\frac{n}{2}]\times[\frac{n+1}{2}]$ or
$[\frac{n+1}{2}]\times [\frac{n}{2}]$ matrix.''
\begin{ex}
\label{B_5-orb}  Let $G=B(5,{\mathbb R}),\,\, {\mathfrak
g}=n_+(5,{\mathbb R}),\,\, {\mathfrak g}^*= n_-(5,{\mathbb R}).$ We
write the representations for generic orbit corresponding to the
point $y=y_{51}E_{51}+y_{42}E_{42}\in {\mathfrak g}^*$. Set ${\mathfrak h}_3=\{t-I\mid t\in
H_3\}$ where
\end{ex}
$$
G=\left\{\left(\begin{smallmatrix}
1&x_{12}&x_{13}&x_{14}&x_{15}\\
0&1&x_{23}&x_{24}&x_{25}\\
0&0&1&x_{34}&x_{35}\\
0&0&0&1&x_{45}\\
0&0&0&0&1\\
\end{smallmatrix}\right)\right\},\,\,\,
H_3=\left\{\left(\begin{smallmatrix}
1&0&0&t_{14}&t_{15}\\
0&1&0&t_{24}&t_{25}\\
0&0&1&t_{34}&t_{35}\\
0&0&0&1&0\\
0&0&0&0&1\\
\end{smallmatrix}\right)\right\}
,\,\,{\mathfrak g}^*=\left\{\left(\begin{smallmatrix}
0&0&0&0&0\\
y_{21}&0&0&0&0\\
y_{31}&y_{32}&0&0&0\\
y_{41}&y_{42}&y_{43}&0&0\\
y_{51}&y_{52}&y_{53}&y_{54}&0\\
\end{smallmatrix}\right)\right\}.
$$
The corresponding representation $S$ of the subgroup $H_3$ of the
maximal dimension  is:
$$
H_3 \ni t\mapsto \exp(2\pi i\langle y_,(t-I)\rangle)=\exp(2\pi
i[t_{15}y_{51}+t_{24}y_{42}])\in S^1.
$$
For the group $B(5,{\mathbb R})$ holds the following decomposition
\begin{equation}
\label{B(5,R)=decom} B(5,{\mathbb
R})=B_{3}B(3)B^{(3)}\text{\,\,i.e.\,\,}x=x_{3}x(3)x^{(3)},
\end{equation}
where
$$
 B^{(3)} =\left\{\left(\begin{smallmatrix}
1&x_{12}&x_{13}&0&0\\
0&1&x_{23}&0&0\\
0&0&1&0&0\\
0&0&0&1&0\\
0&0&0&0&1\\
\end{smallmatrix}\right)\right\}
,\,\, B(3)=\left\{\left(\begin{smallmatrix}
1&0&0&x_{14}&x_{15}\\
0&1&0&x_{24}&x_{25}\\
0&0&1&x_{34}&x_{35}\\
0&0&0&1&0\\
0&0&0&0&1\\
\end{smallmatrix}\right)\right\}
,\,\, B_3=\left\{\left(\begin{smallmatrix}
1&0&0&0&0\\
0&1&0&0&0\\
0&0&1&0&0\\
0&0&0&1&x_{45}\\
0&0&0&0&1\\
\end{smallmatrix}\right)\right\}.
$$
We {\it calculate $h(x,t)$ in the relation} $s(x)t=h(x,t)s(xt)$, but first we fix the section $s:X=H\backslash G\mapsto G$ of the projection
$p:G\mapsto X$. To define the section $s:X\mapsto G$ we show that in addition to the decomposition (\ref{B(5,R)=decom})
the following decomposition $B(5,{\mathbb R})=B(3)B_{3}B^{(3)}$ also holds. Indeed, to find $h\in H_3=B(3)$
such that $x=hx_{3}x^{(3)}$, we get $x_{3}x(3)x^{(3)}=hx_{3}x^{(3)}$, hence
$$
h=x_{3}x(3)x_{3}^{-1}=
\left(\begin{smallmatrix}
1&0&0&x_{14}&x_{15}\\
0&1&0&x_{24}&x_{25}\\
0&0&1&x_{34}&x_{35}\\
0&0&0&1&x_{45}\\
0&0&0&0&1\\
\end{smallmatrix}\right)\left(\begin{smallmatrix}
1&0&0&0&0\\
0&1&0&0&0\\
0&0&1&0&0\\
0&0&0&1&-x_{45}\\
0&0&0&0&1\\
\end{smallmatrix}\right)=
\left(\begin{smallmatrix}
1&0&0&x_{14}&x_{15}-x_{14}x_{45}\\
0&1&0&x_{24}&x_{25}-x_{24}x_{45}\\
0&0&1&x_{34}&x_{35}-x_{34}x_{45}\\
0&0&0&1&0\\
0&0&0&0&1\\
\end{smallmatrix}\right)
\in B(3).
$$
We have two different decompositions
$$
B_{3}B(3)B^{(3)}\ni x_{3}x(3)x^{(3)}=hx_{3}x^{(3)}\in B(3)B_{3}B^{(3)},\,\,\,\,\text{with}\,\,\,\,
h=x_{3}x(3)x_{3}^{-1}.
$$
\begin{rem}
\label{9.B(n,R)=LAU}
For an arbitrary $n,\,\,m\in {\mathbb N},\,\,1<m<n,$ we have
for the group $G_n=B(n,\mathbb R)$ two decompositions:
\begin{equation}
\label{2decompos} G_n\!=\!B_{m}B(m)B^{(m)}\ni
x_{m}x(m)x^{(m)}=hx_{m}x^{(m)}\in B(m)B_{m}B^{(m)},\,\,
h\!=\!x_{m}x(m)x_{m}^{-1},
\end{equation}
where
$$
B_{m}=\{I+\!\!\sum_{m< k<r\leq n}x_{kr}E_{kr}\},\,\,B(m)=
\{I+\!\!\sum_{1\leq k\leq m <r\leq n}x_{kr}E_{kr}\},\,\,
B^{(m)}=\{I+\!\!\sum_{1\leq k<r\leq m}x_{kr}E_{kr}\}.
$$
\end{rem}
Since $X=B(m)\backslash G_n$ is isomorphic to $B_{m}B^{(m)}$ by
decomposition (\ref{2decompos}), the section $s$ can be choosing,
by Remark~\ref{9.r.s:X-G}, as  the embedding
$$
B_{m}B^{(m)}\ni x_mx^{(m)}\mapsto s(x_mx^{(m)})=x_mx^{(m)}\in B_{m}B(m)B^{(m)}.
$$
Since $s(x)t=h(x,t)s(xt)$, we have $h(x,t)=s(x)t(s(xt))^{-1}$. It
remains to calculate $s(x)t$ and $s(xt)$.
\begin{rem}
\label{h(x,t)=e} We have
$$
h(x,t)-I=\left\{\begin{array}{ccc}
0,&\text{for}& t\in B_{m}B^{(m)}\\
x^{(m)}(t-I)x_m^{-1}, &\text{for}& t\in B(m)
\end{array}\right..
$$
\end{rem}
Indeed, let $t=t_{m}t^{(m)}\in B_{m}B^{(m)}$ then
$s(x)t=x_{m}x^{(m)}t_{m}t^{(m)}=x_{m}t_{m}x^{(m)}t^{(m)}$. We get
also $xt=x_{m}x^{(m)}t_{m}t^{(m)}=x_{m}t_{m}x^{(m)}t^{(m)}$, so
$s(xt)\!=\!x_{m}t_{m}x^{(m)}t^{(m)}$, hence $s(x)t\!=\!s(xt)$ and
we get $h(x,t)\!=\!e$. For $t:=t(m)\in B(m)$ and $x=x_mx^{(m)}\in
B_mB(m)$  we get
$$
s(x)t=x_mx^{(m)}t=x_mx^{(m)}t(x^{(m)})^{-1}x^{(m)}=x_m{\tilde x}(m)x^{(m)}=hx_mx^{(m)}=h(x,t)s(xt),
$$
where ${\tilde x}(m)=x^{(m)}t(x^{(m)})^{-1}$. Then we get by (\ref{2decompos})
\begin{equation}
\label{h(x,t)}
h(x,t)=h=x_m{\tilde x}(m)x_m^{-1}=x_mx^{(m)}t(x^{(m)})^{-1}x_m^{-1}=x_mx^{(m)}t(x_mx^{(m)})^{-1},
\end{equation}
\begin{equation}
\label{h(x,t)-m}
h(x,t)=
\left(\begin{smallmatrix}
x^{(m)}&0\\
0&x_m\\
\end{smallmatrix}\right)
\left(\begin{smallmatrix}
1&t-I\\
0&1\\
\end{smallmatrix}\right)
\left(\begin{smallmatrix}
(x^{(m)})^{-1}&0\\
0&x_m^{-1}\\
\end{smallmatrix}\right)
=
\left(\begin{smallmatrix}
1&x^{(m)}(t-I)x_m^{-1}\\
0&1\\
\end{smallmatrix}\right)=
\left(\begin{smallmatrix}
1&H(x,t)\\
0&1\\
\end{smallmatrix}\right),
\end{equation}
where
\begin{equation}
\label{H(x,t)} H(x,t):= x^{(m)}(t-I)x_m^{-1}.
\end{equation}
Denote by   $E_{kr}(t):=I+tE_{kr},\,\,t\in{\mathbb R}$ the one-parameter subgroups of the groups
$B(n,{\mathbb R})$. We would like {\it to find the generators $A_{kn}=\frac{d}{dt}T_{I+tE_{kn}}\vert_{t=0}$
of the induced representation} $T_t$ (\ref{Ind(B_n)}).

Set for $G_n=\!B_{m}B(m)B^{(m)}$ and $1\leq k\leq m< r\leq n$
\begin{equation}
\label{S_kr(t)}
 S_{kr}(t_{kr}):=\langle
y,(h(x,E_{kr}(t_{kr}))-I)\rangle,\,\,\text{then}\,\,
A_{kr}=\frac{d}{dt}\exp(2\pi i S_{kr}(t))\vert_{t=0}=2\pi i S_{kr}(1).
\end{equation}
Let us denote by ${\mathbb S}$ the following matrix:
\begin{equation}
\label{def:S=(S_kr)}
{\mathbb S}=(S_{kr})_{1\leq k\leq m< r\leq n},\quad \text{where}\quad S_{kr}=S_{kr}(1),
\quad \text{then}\quad {\mathbb S}=(2\pi i)^{-1}(A_{kr})_{k,r}.
\end{equation}
\begin{lem}
\label{l.tr(E_knB)=B^t} Let  $B=(b_{kr})_{k,r=1}^n\in {\rm
Mat}(n,{\mathbb C})$. Define the matrix $C=(c_{kr})_{k,r=1}^n\in
{\rm Mat}(n,{\mathbb C})$ by
\begin{equation}
\label{tr(E_knB)=B^t}
 c_{kr}={\rm tr}(E_{kr}B),\quad 1\leq k,r\leq n,\quad
\text{then we have}\quad C=B^T,
\end{equation}
where $E_{kr}$ are matrix units and  $B^T$ means transposed matrix to the matrix $B$. The equality $C=B^T$ holds also in the case when
$B$ is an arbitrary $m\times n$ rectangular matrix. The statement is true also for matrices $B \in {\rm
Mat}(\infty,{\mathbb C})$.
\end{lem}
\begin{proof} Indeed, we have ${\rm tr}(E_{kr}B)=b_{rk}$.
\end{proof}

We calculate now the matrix  ${\mathbb S}(t)=(S_{kr}(t_{kr}))_{k,r}$ and the matrix ${\mathbb S}=
(S_{kr}(1))_{k,r}$ using Lemma~\ref{l.tr(E_knB)=B^t}. Using (\ref{H(x,t)}) we have
$$
\langle y,h(x,t)-I\rangle={\rm tr}\left(H(x,t)y\right)={\rm tr}\left(x^{(m)}t_0x_m^{-1}y\right)=
{\rm tr}\left(t_0x_m^{-1}yx^{(m)}\right)={\rm tr}\left(t_0B(x,y)\right),
$$
where $t_0=t-I$ and
\begin{equation}
\label{B(x,y)}
B(x,y)=x_m^{-1}yx^{(m)}\cong
\left(\begin{smallmatrix}
1&0\\
0&x_m^{-1}\\
\end{smallmatrix}\right)
\left(\begin{smallmatrix}
0&0\\
y&0\\
\end{smallmatrix}\right)
\left(\begin{smallmatrix}
x^{(m)}&0\\
0&1\\
\end{smallmatrix}\right)=
\left(\begin{smallmatrix}
0&0\\
x_m^{-1}yx^{(m)}&0\\
\end{smallmatrix}\right).
\end{equation}
By definition we have
$$
S_{kr}(t_{kr})=\langle y,(h(x,E_{kr}(t_{kr}))-I)\rangle={\rm tr}(t_{kr}E_{kr}B(x,y)),
$$
hence by Lemma~\ref{l.tr(E_knB)=B^t} and (\ref{B(x,y)}) we conclude that
\begin{equation}
\label{S=B^T}
{\mathbb S}=(S_{kr}(1))_{kr}=\left({\rm tr}\left(E_{kr}B(x,t)\right)\right)_{k,r}=B^T(x,y)=
(x^{(m)})^Ty^T(x_m^{-1})^T=
\left(\begin{smallmatrix}
0&(x^{(m)})^Ty^T(x_m^{-1})^T\\
0&0\\
\end{smallmatrix}\right).
\end{equation}
So the  induced
representation ${\rm Ind}^G_H(S):G\to U(L^2(X,\mu))$ corresponding
to the point $y\in {\mathfrak g}^*$ has
the following form
\begin{equation}\label{Ind(B_n)}
(T_tf)(x)\!=\!S(h(x,t))\left(\frac{d\mu(xt)}{d\mu(x)}\right)^{1/2}\!
\!\!f(xt),\,\,f\in L^2(X,\mu),\,\,x\in
X=H\backslash G,\,t\in G,
\end{equation}
where
\begin{equation}
\label{9.S(h)=exp(2pi tr(B))} S(h(x,t))=\exp( 2\pi i\langle y,(h(x,t)-I)\rangle)=\exp\Big(2\pi i
{\rm tr}\left((t-I)B(x,y)\right) \Big).
\end{equation}
We {\it calculate $B(x,y)$ and ${\mathbb S}$ for different groups}
$G_n$. For $G_5$ we get by (\ref{B(x,y)}):
$$
G_5=\left\{\left(\begin{smallmatrix}
1&x_{12}&x_{13}&x_{14}&x_{15}\\
0&1&x_{23}&x_{24}&x_{25}\\
0&0&1&x_{34}&x_{35}\\
0&0&0&1&x_{45}\\
0&0&0&0&1\\
\end{smallmatrix}\right)\right\},\,\,
y=\left(\begin{smallmatrix}
0&0&0&0&0\\
0&0&0&0&0\\
0&0&0&0&0\\
0&y_{42}&0&0&0\\
y_{51}&0&0&0&0\\
\end{smallmatrix}\right),\,\,
x^{(3)}=\left(\begin{smallmatrix}
1&x_{12}&x_{13}\\
0&1&x_{23}\\
0&0&1\\
\end{smallmatrix}\right),\,\,
x_3=
\left(\begin{smallmatrix}
1&x_{45}\\
0&1\\
\end{smallmatrix}\right),
$$
$$
B(x,y)=
\left(\begin{smallmatrix}
1&x_{45}^{-1}\\
0&1\\
\end{smallmatrix}\right)
\left(\begin{smallmatrix}
0&y_{42}&0\\
y_{51}&0&0\\
\end{smallmatrix}\right)
\left(\begin{smallmatrix}
1&x_{12}&x_{13}\\
0&1&x_{23}\\
0&0&1\\
\end{smallmatrix}\right)=
\left(\begin{smallmatrix}
x_{45}^{-1}y_{51}&y_{42}+x_{45}^{-1}y_{51}x_{12}&y_{42}x_{23}+x_{45}^{-1}y_{51}x_{13}\\
y_{51}& y_{51}x_{12}& y_{51}x_{13} \\
\end{smallmatrix}\right),
$$
hence by (\ref{S=B^T}) we have
\begin{equation}
\label{S_5}
{\mathbb S}:=B(x,y)^T= \left(\begin{smallmatrix}
1&0&0\\
x_{12}&1&0\\
x_{13}&x_{23}&1\\
\end{smallmatrix}\right)
\left(\begin{smallmatrix}
0&y_{51}\\
y_{42}&0\\
0&0\\
\end{smallmatrix}\right)
\left(\begin{smallmatrix}
1&x_{45}^{-1}\\
0&1\\
\end{smallmatrix}\right)=\left(\begin{smallmatrix}
x_{45}^{-1}y_{51}&y_{51}\\
y_{42}+x_{45}^{-1}y_{51}x_{12}&y_{51}x_{12}\\
y_{42}x_{23}+x_{45}^{-1}y_{51}x_{13}&y_{51}x_{13}\\
\end{smallmatrix}\right).
\end{equation}
\begin{rem}\label{x^{-1}:=}
For the matrix $x=I+\sum_{1\leq k<n\leq m}x_{kn}E_{kn}\in
B(m,{\mathbb R})$ we denote by $x_{kn}^{-1}$ the matrix elements of
the matrix $x^{-1}$, i.e. $ x^{-1}=:I+\sum_{1\leq k<n\leq
m}x_{kn}^{-1}E_{kn}\in B(m,{\mathbb R}). $ The explicit expressions
for $x^{-1}_{kn}$ are as follows (see \cite{Kos88}, formula (4.4))
$x^{-1}_{kk+1}=-x_{kk+1}$,
\begin{equation}
\label{x^{-1}_{kn}:=}
x_{kn}^{-1}=-x_{kn}+\sum_{r=1}^{n-k-1}(-1)^{r-1}
\sum_{k<i_1<i_2<...<i_r<n}x_{ki_1}x_{i_1i_2}...x_{i_rn},\,\,k<n-1.
\end{equation}
\end{rem}
{\it The generators} $A_{kn}=\frac{d}{dt}T_{I+tE_{kn}}\vert_{t=0}$
of the one-parameter subgroups
$E_{kn}(t):=I+tE_{kn},\,\,t\in{\mathbb R}$ generated by {\it the
representation} $T_t$ (\ref{Ind(B_n)}) are as follows (see
(\ref{def:S=(S_kr)}) and (\ref{S_5})):
\begin{equation}
\label{9.A_kn-B-5}
A_{12}=D_{12},\quad A_{13}=D_{13},\quad A_{23}=x_{12}D_{13}+D_{23}, \quad A_{45}=D_{45},
\end{equation}
\begin{equation}
\label{9.A_kn-B-5(3)}
{\mathbb S}=
\frac{1}{2\pi i }\left(\begin{smallmatrix}
A_{14}&A_{15}\\
A_{24}&A_{25}\\
A_{34}&A_{35}\\
\end{smallmatrix}\right)=
\left(\begin{smallmatrix}
x_{45}^{-1}y_{51}&y_{51}\\
y_{42}+x_{45}^{-1}y_{51}x_{12}&y_{51}x_{12}\\
y_{42}x_{23}+x_{45}^{-1}y_{51}x_{13}&y_{51}x_{13}\\
\end{smallmatrix}\right),
\end{equation}
where $D_{kn}=\frac{\partial}{\partial x_{kn}}$. For example, to obtain the expression $A_{23}=x_{12}D_{13}+D_{23}$ we note that
$$
B(3,{\mathbb R})\ni x(I+tE_{23})=
\left(\begin{smallmatrix}
1&x_{12}&x_{13}\\
0&1&x_{23}\\
0&0&1\\
\end{smallmatrix}\right)
\left(\begin{smallmatrix}
1&0&0\\
0&1&t\\
0&0&1\\
\end{smallmatrix}\right)=
\left(\begin{smallmatrix}
1&x_{12}&x_{13}+tx_{12}\\
0&1&x_{23}+t\\
0&0&1\\
\end{smallmatrix}\right).
$$
Here we denote by $D_{kn}=D_{kn}(h)$ the operator of the partial
derivative corresponding to the shift $x\mapsto x+tE_{kn}$  on the
group $B_m\times B^{(m)}\ni x=(x_{kn})_{k,n}$ and the Haar measure
$h$:
\begin{equation}
\label{D_{kn}(h)}
(D_{kn}(h)f)(x)=\frac{d}{dt}\left(\frac{dh(x+tE_{kn})}{dh(x)}\right)^{1/2}f(x+tE_{kn})\mid_{t=0},\quad
D_{kn}(h):=\frac{\partial}{\partial x_{kn}}.
\end{equation}
\begin{ex}
\label{B_4-orb}  Let $G=B(4,{\mathbb R})=\left\{\left(\begin{smallmatrix}
1&x_{23}&x_{24}&x_{25}\\
0&1     &x_{34}&x_{35}\\
0&0     &1&x_{45}\\
0&0     &0&1\\
\end{smallmatrix}\right)\right\}$. The representations for generic orbit corresponding to the
point $y=y_{43}E_{43}+y_{52}E_{52}\in {\mathfrak g}^*$.
\end{ex}
We calculate ${\mathbb S}$ in two different ways. First using (\ref{B(x,y)}) we get
$$
B(x,y)=x_m^{-1}yx^{(m)}=
\left(\begin{smallmatrix}
1&x_{45}^{-1}\\
0&1\\
\end{smallmatrix}\right)
\left(\begin{smallmatrix}
0&y_{43}\\
y_{52}&0\\
\end{smallmatrix}\right)
\left(\begin{smallmatrix}
1&x_{23}\\
0&1\\
\end{smallmatrix}\right)=
\left(\begin{smallmatrix}
x_{45}^{-1}y_{52}&y_{43}+x_{45}^{-1}y_{52}x_{23}\\
y_{52}&x_{23}y_{52}\\
\end{smallmatrix}\right),
$$
$$
\frac{1}{2\pi i}\left(\begin{smallmatrix}
A_{24}&A_{25}\\
A_{34}&A_{35}\\
\end{smallmatrix}\right)=
{\mathbb S}= B^T(x,y)=\left(\begin{smallmatrix}
1&0\\
x_{23}&1\\
\end{smallmatrix}\right)
\left(\begin{smallmatrix}
0&y_{52}\\
y_{43}&0\\
\end{smallmatrix}\right)
\left(\begin{smallmatrix}
1&0\\
x_{45}^{-1}&1\\
\end{smallmatrix}\right)=\left(\begin{smallmatrix}
x_{45}^{-1}y_{52}&y_{52}\\
y_{43}+x_{45}^{-1}y_{52}x_{23}&y_{52}x_{23}\\
\end{smallmatrix}\right),
$$
$$
A_{23}=D_{23},\quad A_{45}=D_{45}.
$$
From the other hand,  by (\ref{h(x,t)-m}) we get
$h(x,t)=\left(\begin{smallmatrix}
1&H(x,\,t) \\
0&1\\
\end{smallmatrix}\right)$,
where
\begin{equation}
\label{H(x,t),3} H(x,t)\!=\!x^{(3)}(t\!-\!I)x_3^{-1}\!=\!\left(\begin{smallmatrix}
1&x_{23} \\
0&1\\
\end{smallmatrix}\right)
\left(\begin{smallmatrix}
t_{24}&t_{25}\\
t_{34}&t_{35}\\
\end{smallmatrix}\right)
\left(\begin{smallmatrix}
1&x_{45}^{-1} \\
0&1\\
\end{smallmatrix}\right)\!=\!
\left(\begin{smallmatrix}
t_{24}+x_{23}t_{34}&(t_{24}+x_{23}t_{34})x_{45}^{-1}+t_{25}+x_{23}t_{35}\\
t_{34}&t_{34}x_{45}^{-1}+t_{25}+t_{35}\\
\end{smallmatrix}\right).
\end{equation}
Therefore,
$$
\langle y,(h(x,t)-I)\rangle=h(x,t)_{34}y_{43}+h(x,t)_{25}y_{52}=
t_{34}y_{43}+[(t_{24}+x_{23}t_{34})x_{45}^{-1}+t_{25}+x_{23}t_{35}]y_{52},
$$
hence
$$
{\mathbb S}_2(t):=
\left(\begin{smallmatrix}
S_{24}(t_{24})&S_{25}(t_{25})\\
S_{34}(t_{34})&S_{35}(t_{35})\\
\end{smallmatrix}\right)=
\left(\begin{smallmatrix}
t_{24}x_{45}^{-1}y_{52}&t_{25}y_{52}\\
t_{34}y_{43}+x_{23}t_{34}x_{45}^{-1}y_{52}&x_{23}t_{35}y_{52}\\
\end{smallmatrix}\right),
$$
\begin{equation}
\label{S_2} {\mathbb S}_2:={\mathbb
S}_2(1)=\left(\begin{smallmatrix}
S_{24}&S_{25}\\
S_{34}&S_{35}\\
\end{smallmatrix}\right)=
\left(\begin{smallmatrix}
x_{45}^{-1}y_{52}&y_{52}\\
y_{43}+x_{45}^{-1}y_{52}x_{23}&y_{52}x_{23}\\
\end{smallmatrix}\right)
=
\left(\begin{smallmatrix}
1&0\\
x_{23}&1\\
\end{smallmatrix}\right)
\left(\begin{smallmatrix}
0&y_{52}\\
y_{43}&0\\
\end{smallmatrix}\right)
\left(\begin{smallmatrix}
1&0\\
x_{45}^{-1}&1\\
\end{smallmatrix}\right).
\end{equation}
\begin{ex}
\label{B_6-orb}  Let $G=B(6,{\mathbb R}),\,\, {\mathfrak
g}=n_+(6,{\mathbb R}),\,\, {\mathfrak g}^*= n_-(6,{\mathbb R}).$
We write the representations for generic orbit corresponding to
the point $y=y_{43}E_{43}+y_{52}E_{52}+y_{61}E_{61}\in {\mathfrak
g}^*$. Set
\end{ex}
$$
G_6=\left\{\left(\begin{smallmatrix}
1&x_{12}&x_{13}&x_{14}&x_{15}&x_{16}\\
0&1&x_{23}&x_{24}&x_{25}&x_{26}\\
0&0&1&x_{34}&x_{35}&x_{36}\\
0&0&0&1&x_{45}&x_{46}\\
0&0&0&0&1&x_{56}\\
0&0&0&0&0&1\\
\end{smallmatrix}\right)\right\},\,\,\,
H_3=\left\{\left(\begin{smallmatrix}
1&0&0&t_{14}&t_{15}&t_{16}\\
0&1&0&t_{24}&t_{25}&t_{26}\\
0&0&1&t_{34}&t_{35}&t_{36}\\
0&0&0&1&0&0\\
0&0&0&0&1&0\\
0&0&0&0&0&1\\
\end{smallmatrix}\right)\right\}
,\,\, y=\left(\begin{smallmatrix}
0&0&0&0&0&0\\
0&0&0&0&0&0\\
0&0&0&0&0&0\\
0&0&y_{43}&0&0&0\\
0&y_{52}&0&0&0&0\\
y_{61}&0&0&0&0&0\\
\end{smallmatrix}\right),
$$
${\mathfrak h}_3=\{t-I\mid t\in H_3\}$. The corresponding
representations $S$ of the subgroup $H_3$ is:
$$
H_3 \ni \exp(t-I)=t\mapsto \exp(2\pi i\langle y,(t-I)\rangle)=\exp(2\pi
i[t_{34}y_{43}+t_{25}y_{52}+t_{16}y_{61}])\in S^1.
$$
For the group $B(6,{\mathbb R})$ holds the following decomposition
(see Remark \ref{9.B(n,R)=LAU})
\begin{equation}
\label{B(6,R)=decom} B(6,{\mathbb
R})=B_{3}B(3)B^{(3)}\quad\text{i.e.}\quad x=x_{3}x(3)x^{(3)},
\end{equation}
where
$$
 x^{(3)} =\left(\begin{smallmatrix}
1&x_{12}&x_{13}&0&0&0\\
0&1&x_{23}&0&0&0\\
0&0&1&0&0&0\\
0&0&0&1&0&0\\
0&0&0&0&1&0\\
0&0&0&0&0&1\\
\end{smallmatrix}\right)
,\,\, x(3)=\left(\begin{smallmatrix}
1&0&0&x_{14}&x_{15}&x_{16}\\
0&1&0&x_{24}&x_{25}&x_{26}\\
0&0&1&x_{34}&x_{35}&x_{36}\\
0&0&0&1&0&0\\
0&0&0&0&1&0\\
0&0&0&0&0&1\\
\end{smallmatrix}\right)
,\,\, x_3=\left(\begin{smallmatrix}
1&0&0&0&0&0\\
0&1&0&0&0&0\\
0&0&1&0&0&0\\
0&0&0&1&x_{45}&x_{46}\\
0&0&0&0&1&x_{56}\\
0&0&0&0&0&1\\
\end{smallmatrix}\right).
$$
We get by (\ref{B(x,y)}) and (\ref{S=B^T})
$$
B(x,y)=
\left(\begin{smallmatrix}
1&x_{45}^{-1}&x_{46}^{-1}\\
0&1&x_{56}^{-1}\\
0&0&1\\
\end{smallmatrix}\right)
\left(\begin{smallmatrix}
0&0&y_{43}\\
0&y_{52}&0\\
y_{61}&0&0\\
\end{smallmatrix}\right)
\left(\begin{smallmatrix}
1&x_{12}&x_{13}\\
0&1&x_{23}\\
0&0&1\\
\end{smallmatrix}\right)
$$
$$
=\left(\begin{smallmatrix}
x_{46}^{-1}y_{61}&x_{45}^{-1}y_{52}+x_{46}^{-1}y_{61}x_{12}&y_{43}+
x_{45}^{-1}y_{52}x_{23}+x_{46}^{-1}y_{61}x_{13}\\
x_{56}^{-1}y_{61}&y_{52}+x_{56}^{-1}y_{61}x_{12}&y_{52}x_{23}+x_{56}^{-1}y_{61}x_{13}\\
y_{61}&y_{61}x_{12}&y_{61}x_{13}\\
\end{smallmatrix}\right),
$$
hence
$$
{\mathbb S}=B^T(x,y)=\left(\begin{smallmatrix}
1&0&0\\
x_{12}&1&0\\
x_{13}&x_{23}&1\\
\end{smallmatrix}\right)
\left(\begin{smallmatrix}
0&0&y_{61}\\
0&y_{52}&0\\
y_{43}&0&0\\
\end{smallmatrix}\right)
\left(\begin{smallmatrix}
1&0&0\\
x_{45}^{-1}&1&0\\
x_{46}^{-1}&x_{56}^{-1}&1\\
\end{smallmatrix}\right)
$$
$$
=\left(\begin{smallmatrix}
x_{46}^{-1}y_{61}&x_{56}^{-1}y_{61}&y_{61}\\
x_{45}^{-1}y_{52}+x_{46}^{-1}y_{61}x_{12}&y_{52}+x_{56}^{-1}y_{61}x_{12}&y_{61}x_{12}\\
y_{43}+x_{45}^{-1}y_{52}x_{23}+x_{46}^{-1}y_{61}x_{13}&
y_{52}x_{23}+x_{56}^{-1}y_{61}x_{13}&y_{61}x_{13}\\
\end{smallmatrix}\right).
$$
Using again (\ref{def:S=(S_kr)}), (\ref{Ind(B_n)}) and
Remark~\ref{h(x,t)=e} we get the following expressions for the
generators $A_{kn}=\frac{d}{dt}T_{I+tE_{kn}}\vert_{t=0}$ of
one-parameter subgroups $I+tE_{kn},\,\,t\in{\mathbb R}$:
\begin{eqnarray}
\label{9.A_kn-B-6}
A_{12}=D_{12},\,\,A_{13}=D_{13},\,\,A_{23}=x_{12}D_{13}+D_{23}, \\
A_{45}=D_{45},\,\,A_{46}=D_{46},\,\,A_{56}=x_{45}D_{46}+D_{56},
\end{eqnarray}
\begin{equation}
\label{9.A_kn-B-6(3)}
{\mathbb S}=\frac{1}{2\pi i}\left(\begin{smallmatrix}
A_{14}&A_{15}&A_{16}\\
A_{24}&A_{25}&A_{26}\\
A_{34}&A_{35}&A_{36}\\
\end{smallmatrix}\right)=
\left(\begin{smallmatrix}\
x_{46}^{-1}y_{61}&x_{56}^{-1}y_{61}&y_{61}\\
x_{45}^{-1}y_{52}+x_{46}^{-1}y_{61}x_{12}&y_{52}+x_{56}^{-1}y_{61}x_{12}&y_{61}x_{12}\\
y_{43}+x_{45}^{-1}y_{52}x_{23}+x_{46}^{-1}y_{61}x_{13}&
y_{52}x_{23}+x_{56}^{-1}y_{61}x_{13}&y_{61}x_{13}\\
\end{smallmatrix}\right).
\end{equation}
We {\it recall the expressions} for  $B(x,y)$  and hence for ${\mathbb S}=B(x,y)^T$ {\it for small}
$n$. For $n=4$ we have
$$
B(x,y)=x_m^{-1}yx^{(m)}=
\left(\begin{smallmatrix}
1&x_{45}^{-1}\\
0&1\\
\end{smallmatrix}\right)
\left(\begin{smallmatrix}
0&y_{43}\\
y_{52}&0\\
\end{smallmatrix}\right)
\left(\begin{smallmatrix}
1&x_{23}\\
0&1\\
\end{smallmatrix}\right)=
\left(\begin{smallmatrix}
x_{45}^{-1}y_{52}&y_{43}+x_{45}^{-1}y_{52}x_{23}\\
y_{52}&y_{52}x_{23}\\
\end{smallmatrix}\right),
$$
$$
{\mathbb S}=\left(\begin{smallmatrix}
1&0\\
x_{23}&1\\
\end{smallmatrix}\right)
\left(\begin{smallmatrix}
0&y_{52}\\
y_{43}&0\\
\end{smallmatrix}\right)
\left(\begin{smallmatrix}
1&0\\
x_{45}^{-1}&1\\
\end{smallmatrix}\right)=\left(\begin{smallmatrix}
x_{45}^{-1}y_{52}&y_{52}\\
y_{43}+x_{45}^{-1}y_{52}x_{23}&y_{52}x_{23}\\
\end{smallmatrix}\right).
$$
For $G^3_2\simeq B(6,{\mathbb R})$ (see (\ref{G^m_n=}) for the notation $G^m_n$) holds:
$$
B(x,y)=
\left(\begin{smallmatrix}
1&x_{45}^{-1}&x_{46}^{-1}\\
0&1&x_{56}^{-1}\\
0&0&1\\
\end{smallmatrix}\right)
\left(\begin{smallmatrix}
0&0&y_{43}\\
0&y_{52}&0\\
y_{61}&0&0\\
\end{smallmatrix}\right)
\left(\begin{smallmatrix}
1&x_{12}&x_{13}\\
0&1&x_{23}\\
0&0&1\\
\end{smallmatrix}\right)
$$
$$
=\left(\begin{smallmatrix}
x_{46}^{-1}y_{61}&x_{45}^{-1}y_{52}+x_{46}^{-1}y_{61}x_{12}&y_{43}+x_{45}^{-1}y_{52}x_{23}+
x_{46}^{-1}y_{61}x_{13}\\
x_{56}^{-1}y_{61}&y_{52}+x_{56}^{-1}y_{61}x_{12}&y_{52}x_{23}+x_{56}^{-1}y_{61}x_{13}\\
y_{61}&y_{61}x_{12}&y_{61}x_{13}\\
\end{smallmatrix}\right)
$$
hence
$$
{\mathbb S}=\left(\begin{smallmatrix}
x_{46}^{-1}y_{61}&x_{56}^{-1}y_{61}&y_{61}\\
x_{45}^{-1}y_{52}+x_{46}^{-1}y_{61}x_{12}&y_{52}+x_{56}^{-1}y_{61}x_{12}&y_{61}x_{12}\\
y_{43}+x_{45}^{-1}y_{52}x_{23}+x_{46}^{-1}y_{61}x_{13}&
y_{52}x_{23}+x_{56}^{-1}y_{61}x_{13}&y_{61}x_{13}\\
\end{smallmatrix}\right)
$$
$$
=\left(\begin{smallmatrix}
1&0&0\\
x_{12}&1&0\\
x_{13}&x_{23}&1\\
\end{smallmatrix}\right)
\left(\begin{smallmatrix}
0&0&y_{61}\\
0&y_{52}&0\\
y_{43}&0&0\\
\end{smallmatrix}\right)
\left(\begin{smallmatrix}
1&0&0\\
x_{45}^{-1}&1&0\\
x_{46}^{-1}&x_{56}^{-1}&1\\
\end{smallmatrix}\right).
$$
For $G^3_3\simeq B(8,{\mathbb R})$  holds:
$$
\left(\begin{smallmatrix}
1&x_{01}&x_{02}&x_{03}&t_{04}&t_{05}&t_{06}&t_{07}\\
0&1     &x_{12}&x_{13}&t_{14}&t_{15}&t_{16}&t_{17}\\
0&0&1&x_{23}  &t_{24}&t_{25}&t_{26}&t_{27}\\
0&0&0&1&t_{34}&t_{35}&t_{36}&t_{37}\\
0&0&0&0  &1&x_{45}^{-1}&x_{46}^{-1}&x_{47}^{-1}\\
0&0&0  &0&0&1&x_{56}^{-1}&x_{57}^{-1}\\
0&0&0  &0&0&0&1&x_{67}^{-1}\\
0&0  &0&0&0&0&0&1\\
\end{smallmatrix}\right),\quad
y=\left(\begin{smallmatrix}
0&0&0&0&0&0&0&0\\
0&0&0&0&0&0&0&0\\
0&0&0&0&0&0&0&0\\
0&0&0&0&0&0&0&0\\
0&0&0&y_{43}  &0&0&0&0\\
0&0&y_{52}  &0&0&0&0&0\\
0&y_{61}&0  &0&0&0&0&0\\
y_{70}&0  &0&0&0&0&0&0\\
\end{smallmatrix}\right).
$$
As before we have
$$
B(x,y)=
\left(\begin{smallmatrix}
1&x_{45}^{-1}&x_{46}^{-1}&x_{47}^{-1}\\
0&1&x_{56}^{-1}&x_{57}^{-1}\\
0&0&1&x_{67}^{-1}\\
0&0&0&1\\
\end{smallmatrix}\right)
\left(\begin{smallmatrix}
0&0&0&y_{43}  \\
0&0&y_{52}  &0\\
0&y_{61}&0  &0\\
y_{70}&0  &0&0\\
\end{smallmatrix}\right)
\left(\begin{smallmatrix}
1&x_{01}&x_{02}&x_{03}\\
0&1     &x_{12}&x_{13}\\
0&0&1&x_{23}  \\
0&0&0&1\\
\end{smallmatrix}\right),
$$
$$
{\mathbb S}=(x^{(m)})^Ty^T(x_m^{-1})^T=\left(\begin{smallmatrix}
1     &0     &0     &0\\
x_{01}&1     &0     &0\\
x_{02}&x_{12}&1     &0\\
x_{03}&x_{13}&x_{23}&1\\
\end{smallmatrix}\right)
\left(\begin{smallmatrix}
0     &0&0&y_{70}\\
0     &0&y_{61}     &0\\
0     &y_{52}&&0\\
y_{43}&0     &0     &0     &\\
\end{smallmatrix}\right)
\left(\begin{smallmatrix}
1            &0          &0&0\\
x_{45}^{-1}  &1          &0&0\\
x_{46}^{-1}  &x_{56}^{-1}&1&0\\
x_{47}^{-1}  &x_{57}^{-1}&x_{67}^{-1}&1\\
\end{smallmatrix}\right).
$$

\subsection{New proof of the irreducibility of the induced representations corresponding to
 a generic orbits}
 \label{new-irred-loc}
\begin{rem}
\label{Kir-irr-max-dim}
By Kirillov's Theorem~\ref{9.t.Kir}  the induced representation $T_{f,H}={\rm Ind}^G_H U_{f,H}$ is
irreducible if and only if the Lie algebra ${\mathfrak h}$ of the
group $H$ is a subalgebra of ${\mathfrak g}$ subordinate to the
functional $f$ with {\it maximal possible dimension}.

The condition of ``maximal possible dimension'' is difficult to
extend for the infinite-dimensional case. That is why in this
section we  give another proof of the irreducibility of the induced
representation of a nilpotent group $B(n,{\mathbb R})$ that will be
extended in Section~\ref{gen-orb,Z,irr} for the infinite-dimensional
analog $B_0^{\mathbb Z}$ of the group $B(n,{\mathbb R})$.
\end{rem}
Let us consider a sequence of a Lie groups $G^m_n$ and its Lie algebras
${\mathfrak g}^m_n,\,\,m\in {\mathbb Z},\,\,n\in{\mathbb N}$ defined as follows
\begin{equation}
\label{G^m_n=}
G^m_n=\{I+\sum_{m-n\leq k<n\leq m+n+1}x_{kn}E_{kn}\},\quad
{\mathfrak g}^m_n=\{\sum_{m-n\leq k<n\leq m+n+1}x_{kn}E_{kn}\}.
\end{equation}
We note that for any $m\in{\mathbb N}$ holds $B_0^{\mathbb Z}=\varinjlim_{n}G^m_n$. We have the
decomposition (see (\ref{9.B(n,R)=LAU}))
\begin{equation*}
\label{G^m_n-decom}
G^m_n=B_{m,n}B(m,n)B^{(m,n)},
\end{equation*}
where
\begin{eqnarray*}
B_{m,n}=\{I+\sum_{(k,r)\in\Delta_{m,n}}x_{kr}E_{kr}\},\,\,
B(m,n)=\{I+\sum_{(k,r)\in\Delta(m,n)}x_{kr}E_{kr}\},\\
B^{(m,n)}=\{I+\sum_{(k,r)\in\Delta^{(m,n)}}x_{kr}E_{kr}\},\,\,
\end{eqnarray*}
and
$$
\Delta(m,n)=\{(k,r)\in{\mathbb Z}^2\mid m-n\leq k\leq m<r\leq m+n+1\},\,\,
$$
$$
\Delta_{m,n}=\{(k,r)\in{\mathbb Z}^2\mid m+1\leq k<r\leq m+n+1\},\,\,
$$
$$
\Delta^{(m,n)}=\{(k,r)\in{\mathbb Z}^2\mid m-n\leq k<r\leq m\}.
$$
The corresponding elements of the group $G^{m}_n$ are as follows
$$
\left(
\begin{smallmatrix}
1&x_{m-n,m-n+1}&\dots&x_{m-n,m-1}&x_{m-n,m}&  t_{m-n,m+1}  &t_{m-n,m+2}  &\dots&t_{m-n,m+n+1}\\
0&1            &\dots&x_{m-n+1,m-1}&x_{m-n+1,m}&t_{m-n+1,m+1}&t_{m-n+1,m+2}&\dots&t_{m-n+1,m+n+1}\\
&              &\dots&&&&                                                  &\dots&\\
0&0&\dots&1&  x_{m-1,m}  &t_{m-1,m+1}  &t_{m-1,m+2}&\dots&t_{m-1,m+n+1}\\
0&0&\dots&0&  1          &t_{m,m+1}    &t_{m,m+2}  &\dots&t_{m,m+n+1}\\
0&0&\dots&0&0&1          &x_{m+1,m+2}  &\dots&x_{m+1,m+n+1}\\
0&0&\dots&0&0&0          &1            &\dots&x_{m+2,m+n+1}\\
&  &\dots&&&&                          &\dots&\\
0&0&\dots&0& 0 &0  &0                  &\dots&x_{m+n,m+n+1}\\
0&0&\dots&0& 0 &0  &0                  &\dots&1            \\
\end{smallmatrix}
\right).
$$
The induced representation of the group $G^m_n$ is defined in the space $L^2(X,d\mu)$ by the following formula
\begin{equation}\label{Ind(G^m_n)}
(T^{m,y_n}_tf)(x)=S(h(x,t))\left(
\frac{d\mu(xt)}{d\mu(x)}
\right)^{1/2}f(xt),\,\,f\in L^2(X,\mu),\,\,\,x\in
X=H\backslash G,\,\,t\in G
\end{equation}
where $X=B(m,n)\backslash G^m_n\cong B_{m,n}\times B^{(m,n)}$ (see (\ref{G^m_n-decom})),
\begin{equation}
\label{Haar-on-G^m_n}
d\mu(x_m,x^{(m)})=dx_m\otimes dx^{(m)}=\otimes_{(k,n)\in\Delta_{m,n}}dx_{kn}\otimes
\otimes_{(k,n)\in\Delta^{(m,n)}}dx_{kn}
\end{equation}
be the Haar measure on the group $B_{m,n}\times B^{(m,n)}$. Denote by
${\mathcal H}^{m,n}=L^2(B_{m,n}\times B^{(m,n)},dx_m\otimes dx^{(m)})$.
\begin{thm}
\label{Ind-fin-irr} The induced representation $T^{m,y_n}$ of the
group $G^m_n$ defined by formula (\ref{Ind(G^m_n)}), corresponding
to generic orbit ${\mathcal O}_{y_n}$, generated by the point  $y_n\in
({\mathfrak g^m_n})^*$,\\ $y_n=\sum_{r=0}^{n-1}
y_{m+r+1,m-r}E_{m+r+1,m-r}$ is irreducible. Moreover the generators
of one-parameter groups
$A_{kr}=\frac{d}{dt}T^{m,y_n}_{I+tE_{kr}}\mid_{t=0}$ are as
follows
$$
A_{kr}=\sum_{s=m-n}^{k-1}x_{ks}D_{rs}+D_{kr},\,\,\,(k,r)\in \Delta^{(m,n)},\quad
A_{kr}=\sum_{s=m+1}^{k-1}x_{ks}D_{rs}+D_{kr},\,\,\, (k,r)\in \Delta_{m,n},
$$
$$
(2\pi i)^{-1}\big(A_{kr}\big)_{(k,r)\in \Delta(m,n)}={\mathbb S}^{(m)}_n=
(S_{kr})_{(k,r)\in \Delta(m,n)}=\big(x_m^{-1}yx^{(m)}\big)^T.
$$
\end{thm}
{\it The irreducibility} of the induced representation of the group $G^m_n$  is based on the following lemma.
\begin{lem}
\label{W(S_kn)=W(x_kn)m-n}
Two von Neumann algebra ${\mathfrak A}^S$ and ${\mathfrak A}^x$  in the space ${\mathcal H}^{m,n}$ generated respectively
by the sets of unitary operators $U_{kr}(t)$ and $V_{kr}(t)$ coincides,
where
\begin{equation}
\label{U_kr(t),V_{kr}(t)-(m,n)}
(U_{kr}(t)f)(x)=\exp(2\pi i S_{kr}(t))f(x),\quad (V_{kr}(t)f)(x):=\exp(2\pi i tx_{kr})f(x),
\end{equation}
$$
{\mathfrak A}^S=\big(U_{kr}(t)=T^{m,y_n}_{I+tE_{kr}}=\exp(2\pi i S_{kr}(t))\mid  t\in {\mathbb R},\,\,(k,r)\in \Delta(m,n)\big)'',
$$
$$
{\mathfrak A}^x=\big(V_{kr}(t):=\exp(2\pi i tx_{kr})\mid t\in
{\mathbb R},\,\, (k,r)\in
\Delta_{m,n}\bigcup\Delta^{(m,n)}\big)''.
$$
\end{lem}
\begin{proof}
Using the decomposition (see  (\ref{B(x,y)}) and (\ref{S=B^T}))
\begin{equation}
\label{S^m_n=}
{\mathbb S}_n^{(m)}=(x_m^{-1}yx^{(m)})^T=(x^{(m)})^Ty^T(x_m^{-1})^T
\end{equation}
we conclude that ${\mathfrak A}^S\subseteq {\mathfrak A}^x$.
Indeed, we get $V_{kr}(t):=\exp(2\pi i tx_{kr})\in {\mathfrak A}^x$ hence the operators $x_{kr}$ of multiplication by the independent
variable $f(x)\mapsto x_{kr}f(x)$ in the space ${\mathcal H}^{m,n}$ are affiliated with the von Neumann algebra ${\mathfrak A}^x$ i.e.
$x_{kr}\,\,\eta\,\,{\mathfrak A}^x$ for $(k,r)\in \Delta_{m,n}\bigcup\Delta^{(m,n)}$.
\begin{df}
Recall {\rm(}c.f. e.g. \cite{Dix69W}{\rm)}
that a non  necessarily bounded self-adjoint operator $A$ in a Hilbert space
$H$ is said to be {\it affiliated} with a von Neumann algebra $M$
of operators in this Hilbert space $H$, if $\exp(itA)\in M$ for
all $t\in{\mathbb R}$. One then writes $A\,\,\eta\,\,M$.
\end{df}
By (\ref{x^{-1}_{kn}:=}) the matrix elements
$x_{kr}^{-1}$ of the matrix $x_m^{-1}\in B_{m,n}$ are also affiliated $x_{kr}^{-1}\,\,\eta\,\,{\mathfrak A}^x$. Using (\ref{S^m_n=})
we conclude that the matrix elements $S_{kr},\,\,\in \Delta(m,n)$ of the matrix ${\mathbb S}_n^{(m)}$ are affiliated:
$S_{kr}\,\,\eta\,\,{\mathfrak A}^x,\,\,(k,r)\in \Delta(m,n)$, so ${\mathfrak A}^S\subseteq {\mathfrak A}^x$.
\index{operator!affiliated}

To prove that ${\mathfrak A}^S\supseteq {\mathfrak A}^x$ we  find
the expressions of the matrix element of the matrix $x^{(m)}\in B^{(m,n)}$ and
$x_m^{-1}\in  B_{m,n}$ in terms of the matrix elements of the matrix ${\mathbb
S}_n^{(m)}=(S_{kr})_{(k,r)\in\Delta(m,n)}$. To do that we connect
the above decomposition
 ${\mathbb S}_n^{(m)}=(x^{(m)})^Ty^T(x_m^{-1})^T$ and the Gaussian
decomposition $C=LDU$ (see Theorem~\ref{t.C=UDL}).
Let us denote by $J$ the $n\times n$ anti-diagonal matrix $J=\sum_{r=0}^{n-1}E_{m-r,m+r+1}$
Using $J^2=I$ and (\ref{S=B^T}) we get
\begin{equation}
\label{SJ=dec2}
{\mathbb S}J= B^T(x,y)J=(x^{(m)})^Ty^T(x_m^{-1})^TJ=(x^{(m)})^T(y^TJ)(J(x_m^{-1})^TJ).
\end{equation}
The latter decomposition (\ref{SJ=dec2}) is in fact the Gauss
decomposition of the matrix ${\mathbb S}J$ i.e. we get
$$
{\mathbb S}J=LDU, \quad\text{where}\quad
L=(x^{(m)})^T,\quad  D=y^TJ,\quad U=J(x_m^{-1})^TJ.
$$
Using the Theorem~\ref{t.C=UDL} we can find the matrix elements of the matrix $x^{(m)}\in B^{(m,n)}$ and $x_m^{-1}\in B_{m,n}$
in terms of the matrix elements of the matrix ${\mathbb S}_n^{(m)}$, hence  we can also find
the matrix elements of the matrix $x_m\in B_{m,n}$. This finish the proof of the lemma.
\end{proof}
We give below the expressions for ${\mathbb S}_nJ$. For $m=3$ and  $n=1$ i.e. for $G^3_1$ we  have
(remind  that $J^2=I$ )
$$
{\mathbb S}_2=
\left(\begin{smallmatrix}
1&0\\
x_{23}&1\\
\end{smallmatrix}\right)
\left(\begin{smallmatrix}
0&y_{52}\\
y_{43}&0\\
\end{smallmatrix}\right)
\left(\begin{smallmatrix}
1&0\\
x_{45}^{-1}&1\\
\end{smallmatrix}\right)=
\left(\begin{smallmatrix}
1&0\\
x_{23}&1\\
\end{smallmatrix}\right)
\left(\begin{smallmatrix}
y_{52}&0\\
0&y_{43}\\
\end{smallmatrix}\right)
\left(\begin{smallmatrix}
x_{45}^{-1}&1\\
1&0\\
\end{smallmatrix}\right),
$$
$$
{\mathbb S}_2J=
\left(\begin{smallmatrix}
1&0\\
x_{23}&1\\
\end{smallmatrix}\right)
\left(\begin{smallmatrix}
y_{52}&0\\
0&y_{43}\\
\end{smallmatrix}\right)
\left(\begin{smallmatrix}
1&x_{45}^{-1}\\
0&1\\
\end{smallmatrix}\right).
$$
For $G^3_2$ we get
$$
{\mathbb S}_3=\left(\begin{smallmatrix}
1&0&0\\
x_{12}&1&0\\
x_{13}&x_{23}&1\\
\end{smallmatrix}\right)
\left(\begin{smallmatrix}
y_{61}&0&0\\
0&y_{52}&0\\
0&0&y_{43}\\
\end{smallmatrix}\right)
\left(\begin{smallmatrix}
x_{46}^{-1}&x_{56}^{-1}&1\\
x_{45}^{-1}&1&0\\
1&0&0\\
\end{smallmatrix}\right),
$$
$$
{\mathbb S}_3J=\left(\begin{smallmatrix}
1&0&0\\
x_{12}&1&0\\
x_{13}&x_{23}&1\\
\end{smallmatrix}\right)
\left(\begin{smallmatrix}
y_{61}&0&0\\
0&y_{52}&0\\
0&0&y_{43}\\
\end{smallmatrix}\right)
\left(\begin{smallmatrix}
1&x_{56}^{-1}&x_{46}^{-1}\\
0&1&x_{45}^{-1}\\
0&0&1\\
\end{smallmatrix}\right).
$$
For $G^3_3$ we have
$$
{\mathbb S}_4=\left(\begin{smallmatrix}
1     &0     &0     &0\\
x_{01}&1     &0     &0\\
x_{02}&x_{12}&1     &0\\
x_{03}&x_{13}&x_{23}&1\\
\end{smallmatrix}\right)
\left(\begin{smallmatrix}
y_{70}&0     &0&0\\
0     &y_{61}&0     &0\\
0     &0     &y_{52}&0\\
0     &0     &0     &y_{43}\\
\end{smallmatrix}\right)
\left(\begin{smallmatrix}
x_{47}^{-1}&x_{57}^{-1}&x_{67}^{-1}&1\\
x_{46}^{-1}&x_{56}^{-1}&1&0\\
x_{45}^{-1}&1          &0&0\\
1          &0          &0&0\\
\end{smallmatrix}\right),
$$
\begin{equation}
\label{S_4J}
{\mathbb S}_4J=\left(\begin{smallmatrix}
1     &0     &0     &0\\
x_{01}&1     &0     &0\\
x_{02}&x_{12}&1     &0\\
x_{03}&x_{13}&x_{23}&1\\
\end{smallmatrix}\right)
\left(\begin{smallmatrix}
y_{70}&0     &0&0\\
0     &y_{61}&0     &0\\
0     &0     &y_{52}&0\\
0     &0     &0     &y_{43}\\
\end{smallmatrix}\right)
\left(\begin{smallmatrix}
1&x_{67}^{-1}&x_{57}^{-1}&x_{47}^{-1}\\
0&1&x_{56}^{-1}&  x_{46}^{-1}\\
0          &0&1&x_{45}^{-1}\\
0          &0          &0&1\\
\end{smallmatrix}\right).
\end{equation}
\begin{proof} {\it of the Theorem~\ref{Ind-fin-irr}}. The irreducibility follows from the Kirillov results
(see Remark~\ref{Kir-irr-max-dim}).
To give another proof of the irreducibility of the induced representation consider the restriction
$T^{m,y_n}\mid_{B(m,n)}$ of this representation to the commutative subgroup $B(m,n)$ of the group
$G^{m}_n$. Note that
$$
{\mathfrak A}^x=\big(\exp(2\pi i tx_{kr})\mid t\in {\mathbb R},\,\,
(k,r)\in \Delta_{m,n}\bigcup\Delta^{(m,n)}\big)''=L^\infty(B_{m,n}\times B^{(m,n)},dx_m\otimes dx^{(m)}).
$$
By Lemma~\ref{W(S_kn)=W(x_kn)m-n} the von Neumann algebra ${\mathfrak A}^S$ generated by this restriction coincides with
$L^\infty(B_{m,n}\times B^{(m,n)},dx_m\otimes dx^{(m)})$. Let now a bounded operator $A$ in a Hilbert
space ${\mathcal H}^{m,n}$ commute with the representation $T^{m,y_n}$. Then $A$ commute by the above arguments
with $L^\infty(B_{m,n}\times B^{(m,n)},dx_m\otimes dx^{(m)})$, therefore the operator $A$ itself is an operator of multiplication by some
essentially bounded function $a\in L^\infty$ i.e. $(Af)(x)=a(x)f(x)$ for
$f\in {\mathcal H}^{m,n}$. Since $A$ commute with the representation $T^{m,y_n}$ i.e. $[A,T^{m,y_n}_t]=0$ for all
$t\in B_{m,n}\times B^{(m,n)}$ we conclude  that
$$
a(x)=a(xt)\,\,({\rm mod}\,\, dx_m\otimes dx^{(m)})\quad \text{for \quad all}\quad t\in B_{m,n}\times B^{(m,n)}.
$$
Since the measure $dh=dx_m\otimes dx^{(m)}$  is the Haar measure on  $G=B_{m,n}\times B^{(m,n)}$, this measure is $G$-right ergodic.
We conclude that $a(x)=const$ $({\rm mod}\, dx_m\otimes dx^{(m)})$.
\end{proof}

\section{Induced representations, infinite-dimensional case}
\label{s.ind-inf}

\subsection{Regular and quasiregular representations of infinite-dimensional groups}
\label{reg+quasireg-rep}
To define the induced representation we explain first
how to define the  regular representation of {\it
infinite-dimensional group} $G$.  Since the initial group in not
locally compact  there is neither Haar (invariant) measure on $G$
(Weil, \cite{Weil53}), nor a $G$-quasi-invariant measure (Xia
Dao-Xing, \cite{{Xia78}}).  We can try to find  some bigger
topological group $\widetilde{G}$ and the $G$-quasi-invariant
measure $\mu$ on $\widetilde{G}$ such that $G$ is the dense
subgroup in $\widetilde{G}.$ In this case we define the {\it right
or left regular representation} of the group $G$ in the space
$L^2(\tilde G, \mu)$ if $\mu^{R_t}\sim \mu$ (resp. $\mu^{L_t}\sim \mu$) for all $t\in G$ as follows:
\begin{equation} \label{T(R,mu)+art}
(T^{R,\mu}_tf)(x)=(d\mu(xt)/d\mu(x))^{1/2}f(xt),\,\,f\in
L^2(\tilde G, \mu),\,\,t\in G,
\end{equation}
\begin{equation}
\label{T(L,mu)+art}
(T^{L,\mu}_tf)(x)=(d\mu(t^{-1}x)/d\mu(x))^{1/2}f(t^{-1}x),\,\,f\in
L^2(\tilde G, \mu),\,\,t\in G.
\end{equation}
\begin{co}
[\rm Ismagilov, 1985] The right regular representation
$T^{R,\mu}:G\rightarrow U(L^2(\tilde G,\mu))$ is irreducible if
and only if\par 1) $\mu^{L_t}\perp \mu\,\,\forall t\in
G\backslash{\{e\}},\,\,$
\par 2) the measure $\mu$ is $G$-ergodic.
\end{co}
Analogously we can define  the {\it quasiregular representation}.
Namely, if $H$ is a closed subgroup of the group $G$, then on the
space $X=\widetilde{H\backslash G}= \tilde H\backslash \tilde G$
the right action of the group $G$ is well defined, where $\tilde
G$ (resp. $\tilde H$) is some completion of the group $G$ (resp.
$H$). If we have some $G$-right-quasi-invariant measure $\mu$ on
$X$ one may define the ``quasiregular representation" of the group
$G$ in the space $L^2(X,\mu)$ as in a locally compact case:
$$
(\pi_t^{R,\mu,X}f)(x)= (d\mu(xt)/d\mu(x))^{1/2}f(xt),\quad t\in G.
$$
The regular and quasiregular representations for general
infinite-dimensional groups were introduced and investigated in e.g.
\cite{KosAlb06J,Kos90,Kos92,Kos94,Kos02.3}.
\subsection{Induced representations for infinite-dimensional groups}
\label{Ind-rep}
The induced representation ${\rm Ind}_H^GS$  of a locally-compact group is the unitary representation
of the group $G$ associated with a unitary representation $S$ of a subgroup $H$ of the group $G$
(see Section~\ref{s.ind-loc-comp}).

As it was mentioned in section~\ref{orb-meth(n)} (see \cite{Kir62, Kir94}) all unitary irreducible
representations up to equivalence $\hat{G_n}$ of the nilpotent group $G_n=B(n,{\mathbb R})$, are
obtained as induced representations ${\rm Ind}^{G_{n}}_H U_{f,H}$
associated with a points $f\in {\mathfrak g}^*_n$ and the
corresponding {\it subordinate} subgroup $H\subset G_n$.
The induced representation ${\rm Ind}^{G_n}_H U_{f,H}$ is defined
canonically in the Hilbert space $L^2(H\backslash G_n,\mu)$.

A.~Kirillov \cite{Kir94}, Chapter I, \S 4, p.10 says: {\it "The
method of induced representations is not directly applicable to
infinite-dimensional groups (or more precisely to a pair $G\supset H
$) with an infinite-dimensional factor $H\backslash G$)".}

Our aim is to {\it develop the concept of induced representations for infinite-dimensi\-onal groups}.
Let we have the infinite-dimensional group $G$ and a unitary representation  $S:H\rightarrow U(V)$ in a Hilbert space $V$ of a subgroup $H$ of
the group $G$ such that the factor space $H\backslash G$ is infinite-dimensional.

In general, it is difficult to construct $G$-quasi-invariant measure on an infinite-dimensional
homogeneous space $H\backslash G$.
As is the case of the regular
and quasiregular
representations of infinite-dimensional groups $G$ (see
Subsection~\ref{reg+quasireg-rep}) it is reasonable to construct
some $G$-quasi-invariant measure on a {\it suitable completion}
$\widetilde{H\backslash G}=\tilde{H}\backslash \tilde{G}$ of the
initial space $H\backslash G$ {\it in a certain topology}, where
$\tilde{H}$ (resp. $\tilde G$) is some completion of the group $H$
(resp. $G$). To go further we should be able to {\it extend the
representation} $ S: H\rightarrow U(V)$ of the group $H$ to the
representation $\tilde S:\tilde H\rightarrow U(V)$ of the completion
$\tilde H$ of the group $H$.

Finally, the induced representation of the group $G$ associated with a unitary representation $S$ of a subgroup $H$
will depend on two  completions $\tilde H$ and $\tilde G$ of the subgroup $H$ and
the group $G$, on an extension $\tilde S:\tilde H\rightarrow U(V)$ of
the representation $S:H\rightarrow U(V)$ and on a choice of the
$G$-quasi-invariant measure $\mu$ on an appropriate completion $\tilde X=\tilde
H\backslash \tilde G$  of the space $H\backslash G$.

Hence the procedure of induction will
not be unique but nevertheless well-defined (if a
$G$-quasi-invariant measure on $\widetilde{H\backslash G}$ exists). So  the uniquely defined
induced representation ${\rm Ind}^G_H S$ in the Hilbert space
$L^2(H\backslash G,V,\mu)$ (in the case of a locally-compact group $G$) should be replaced by the family of induced
representations ${\rm Ind}_{\tilde H,H}^{\tilde G,G,\mu}(\tilde S,S)$
in the Hilbert spaces $L^2(\tilde H\backslash \tilde G,V,\mu)$ depending on
different completions $\tilde G$ of the group $G$, completions $\tilde H$ of the group $H$ and different
$G$-quasi-invariant measures $\mu$ on $\tilde H\backslash\tilde G$.
\begin{ex}[\cite{Kos90,Kos94}]
\label{ex.ind-reg} {\rm Regular representations}  $T^{R,\mu}$ of
the infinite-dimensional group $G$ in the space $L^2(\tilde
G,\mu)$, associated with the completion $\tilde G$ of the group
$G$ and a $G$-right -quasi-invariant measure $\mu$ on $\tilde G$,
is a particular case of the induced representation (see
Remark~\ref{reg=ind})
$$
T^{R,\mu}={\rm Ind}_{e}^{\tilde G,G,\mu}(Id),
$$
generated by the trivial representation $S=Id$  of the trivial subgroup $H=\{e\}$
(as in the case of a locally compact groups).
\end{ex}
\begin{ex} [\cite{KosAlb06J,Kos02.3}]
\label{ex.ind-q-reg}{\rm Quasi-regular representations}
$\pi^{R,\mu,X}$ of the infinite-dimen\-si\-onal group $G$ in the
space $L^2(X,\mu)$ where $X=\tilde H\backslash \tilde G$ and $H$
is some subgroup of the group $G$ is a particular case of the
induced representation (see Remark~\ref{reg=ind})
$$
 \pi^{R,\mu,X}={\rm Ind}_{\tilde H,H}^{\tilde
G,G,\mu}(Id)
$$
generated by the trivial representation $S=Id$ of the {\rm completion}
$\tilde H$ in the group $\tilde G$ of the subgroup $H$ in the group
$G$.
\end{ex}
Let $G$ be an infinite-dimensional group  and $S:H\rightarrow U(V)$ be a unitary representation  in a Hilbert space $V$ of
the subgroup $H\subset G$, such that the space $H\backslash G$ is infinite-dimensional. We  give the following definition.
\begin{df}
\label{df.ind-inf-dim} The induced representation
$$
{\rm Ind}_{\tilde H,H}^{\tilde G,G,\mu}(\tilde S,S),
$$
generated by the unitary representations $S:H\rightarrow U(V)$ of
the subgroup $H$ in the group $G$ is defined (similarly to
(\ref{ex.ind-reg}) and (\ref{ex.ind-q-reg})) as follows:
\par 1) we should first find some completion $\tilde H$ of the group $H$
such that
$$
\tilde S:\tilde H\rightarrow U(V)
$$
is the continuous unitary representation of the group $\tilde H$,
such that $\tilde S\vert_{H}=S$,
\par 2) take any $G$-right-quasi-invariant measure $\mu$ on the   an appropriate completion
$\tilde X=\tilde H\backslash \tilde G$ of the space $X=H\backslash G$, on which the group $G$ acts from
the right, where $\tilde{H}$ (resp. $\tilde G$) is a suitable completion of the group $H$ (resp. $G$),
\par 3) in the space $L^2( \tilde X,V,\mu)$
of all vector-valued functions $f$ on $\tilde X$ with values in $V$ such that
$$
\Vert f\Vert^2:=\int_{\tilde X} \Vert f(x)\Vert^2_Vd\mu(x)<\infty,
$$
define the representation of the group $G$ by the following
formula
\begin{equation}\label{ind-inf-dim}
(T_tf)(x)=S(\tilde
h(x,t))\left(\frac{d\mu(xt)}{d\mu(x)}\right)^{1/2}f(xt),\quad x\in \tilde X,\,\,t\in G,
\end{equation}
where $\tilde h$ is defined by $$\tilde s(x)t=\tilde h(x,t)\tilde
s(xt).$$
The section $s:H\rightarrow G$ of the projection $p:G\rightarrow H$
should be extended to the appropriate section $\tilde s:\tilde
H\rightarrow \tilde G$ of the extended projection $\tilde p:\tilde
G\rightarrow \tilde H$.
\end{df}
The comparison of the induced representation for locally compact
group and the above definition for infinite-dimensional groups may be given in the following table:
\vskip 0.5cm
\begin{tabular}{|p{0.3cm}|p{0.9cm}|p{5.15cm}|p{5.15cm}|}\hline
1&$G$&$G$ loc.comp. &${\rm dim}\,G=\infty$\\   \hline
2&$H$ &$H\subset G$&$H\subset G$\\   \hline
3&$S$ &$S:H\rightarrow U(V)$&$S\!:\!H\!\rightarrow\!
U(V)\!\Rightarrow \!\tilde S:\tilde H\!\rightarrow \!U(V)$\\
\hline
4&$X$ &$X=H\backslash G$&$\tilde X=\widetilde{H\backslash G}=\tilde
H\backslash \tilde G$\\   \hline
5&${\mathcal H}$ &$L^2(X=H\backslash G,V,\mu)$&$L^2(\tilde X=\tilde
H\backslash \tilde G,V,\mu)$\\   \hline
6&${\rm Ind}$ &${\rm Ind}_H^GS$&${\rm Ind}_{\tilde H,H}^{\tilde
G,G,\mu}(\tilde S,S)$\\   \hline
7&$T_t$
&$(T_tf)(x)\!=\!S(h(x,t))(\frac{d\mu(xt)}{d\mu(x)})^{1/2}f(xt)$&$(T_tf)(x)\!=\!
\tilde S(\tilde h(x,t))(\frac{d\mu(xt)}{d\mu(x)})^{1/2}f(xt)$\\
\hline
8&$p$ &$p:G\rightarrow X$&$\tilde p:\tilde G\rightarrow \tilde  X$\\
\hline
9&$s$&$s:X\rightarrow G$&$s:H\backslash G\rightarrow G\Rightarrow
\tilde s:
\widetilde{H\backslash G}\rightarrow \tilde G$\\
\hline
10&$h(x,t)$ &$s(x)t=h(x,t)s(xt)$&$\tilde s(x)t=\tilde h(x,t)\tilde s(xt)$\\
\hline

\end{tabular}

\subsection{How to develop the orbit method for infinite-dimensional
``nilpotent'' group $B_0^{\mathbb N}$ and $B_0^{\mathbb Z}$?}
\label{orb-meth-inf}
\index{representation!induced}
{\it We would like to develop the orbit method for
infinite-dimensional ``nilpotent'' group} $G= \varinjlim_{n}G_n$
with $G_n=B(n,{\mathbb R})$. The corresponding Lie algebra
${\mathfrak g}$ is the {\it inductive limit} ${\mathfrak
g}=\varinjlim_{n}{\mathfrak b}_n$ of upper triangular matrices, so
as the linear space it is isomorphic to the space ${\mathbb
R}_0^\infty$ of finite sequences $(x_k)_{k\in{\mathbb N}}$ hence the
dual space ${\mathfrak g}^*$ is isomorphic to the space ${\mathbb
R}^\infty$ of all sequences $(x_k)_{k\in{\mathbb N}}$, but {\it  the
latter  space ${\mathbb R}^\infty$ is too large to manage  with it,
for example to equip  with a Hilbert structure or to describe all
orbits}. To make it less it is reasonable to increase  the initial
group $G$ or to make completion $\tilde G$ of this group in some
stronger topology.
\index{inductive limit}

{\it   To develop the orbit method
for  groups $B_0^{\mathbb N}$ and $B_0^{\mathbb Z}$} we should answer some questions:
\par (1) How to define the {\it appropriate completion} $\tilde G$ of the group $G$,
corresponding Lie algebras ${\mathfrak
g}$ (resp. $\tilde{\mathfrak
g}$) and corresponding dual spaces ${\mathfrak g}^*$ (resp.$\tilde{\mathfrak g}^*$)?
\par (2) Which {\it pairing} should we use between ${\mathfrak g}$ and ${\mathfrak
g}^*$?
\par (3) Let the dual space ${\mathfrak g}^*$, some element $f\in {\mathfrak g}^*$
 and corresponding algebra ${\mathfrak h}$, subordinate to the element
$f$, are chosen. How to {\it define the corresponding induced
representation} ${\rm Ind}^G_H U_{f,H}$ and {\it study its irreducibility} ?
\par (4) Shall we get all irreducible representations of the corresponding groups, using induced representations?
\par (5) Find the criteria of irreducibility and equivalence  of induced representations.

The problem of {\it completion} of the inductive limit group $G=\varinjlim_{n}G_n$, where $G_n$ are
finite-dimensional classical groups were studied by
A.~Kirillov (\cite{Kir73}, 1972) for the group
$U(\infty)\!=\varinjlim_{n}U(n)$ and G.~Olshanski\u\i (\cite{Ols90},
1990) for inductive limit of classical groups. They described all unitary
irreducible representations of the corresponding groups  $G=\varinjlim_{n}G_n$, {\it continuous}
in stronger topology, namely {\it in the strong operator
topology}. The description of the dual $\hat G$ of the initial group $G=\varinjlim_{n}G_n$ is much
more complicated.

In \cite{Kos88} (see details in section \ref{Hilb-Lie-GL_2(a)}) we have constructed for the group
${\rm GL}_0(2\infty,{\mathbb R})$ $=\varinjlim_{n}{\rm GL}(2n-1,{\mathbb R})$
a family of the Hilbert-Lie groups ${\rm GL}_2(a),\,\,a\in{\mathfrak A}$ such that\\
a) ${\rm GL}_0(2\infty,{\mathbb R})\subset {\rm GL}_2(a)$ and ${\rm GL}_0(2\infty,{\mathbb R})$ is dense in ${\rm GL}_2(a)$
for all $a\in{\mathfrak A}$,\\
b) ${\rm GL}_0(2\infty,{\mathbb R})=\cap_{a\in{\mathfrak A}}{\rm GL}_2(a)$,\\
c) {\it any  continuous representation} of the group ${\rm GL}_0(2\infty,{\mathbb R})$ {\it is in fact
continuous} in some stronger topology, namely {\it in a topology of a suitable Hilbert -Lie group}
${\rm GL}_2(a)$.
\index{topology!strong operator}

(1) Therefore, as we show in Sections~\ref{Hilb-Lie-B_2(a)}, \ref{Hilb-Lie-GL_2(a)}  it is sufficient
to consider a {\it Hilbert-Lie completions} $B_2(a)$ of the initial group $B_0^{\mathbb Z}$.

(2) In this case the {\it pairing} between the corresponding Hilbert-Lie algebra ${\mathfrak b}_2(a)$
and its dual ${\mathfrak b}_2(a)^*$ is correctly defined by the trace (as in the finite-dimensional case).

(3.1) We  define in Section~\ref{gen-orb,Z} {\it the induced representations} of the group $B_0^{\mathbb Z}$
 corresponding  to a special orbits, {\it generic orbits}, using schema   given in Section~\ref{Ind-rep}. We
consider only the simplest example of $G-$quasi-invariant measures on $\tilde X=\tilde H\setminus\tilde G$,
namely the infinite product of one-dimensional Gaussian measures.

(3.2) How to construct the {\it induced representation corresponding
to an arbitrary orbit}?
%
\begin{co}
Two induced representations ${\rm Ind}^{\tilde G,\mu_1}_{H_1}
U_{f_1,H_1}$ and ${\rm Ind}^{\tilde G,\mu_2}_{H_2} U_{f_2,H_2}$ are
equivalent if and only if the corresponding measures $\mu_1$ and
$\mu_2$ are equivalent and the functionals $f_1$ and $f_2$ belong to
the same orbit of $(\tilde{\mathfrak g})^*$.
\end{co}
\subsection{Hilbert-Lie groups ${\rm GL}_2(a)$}\label{Hilb-Lie-GL_2(a)}
We show that the {\it Hilbert-Lie groups} appear
naturally in the representation theory of infinite-dimensional
matrix group.
The remarkable fact is that for the inductive limit
$G=\varinjlim_{n}G_n$ of matrix groups $G_n\subset {\rm
GL}(2n-1,{\mathbb R})$ it is sufficient to consider only the {\it
Hilbert completions} of the initial group $G$ and of the spaces
$H\backslash G$.

Let us consider the group ${\rm GL}_0(2\infty,{\mathbb R})=\varinjlim_{n}{\rm
GL}(2n-1,{\mathbb R})$ with respect to the symmetric embedding
$i^s_n:G_n\mapsto G_{n+1}$, $G_n\ni x \mapsto x+E_{-n,-n}+E_{nn}\in
G_{n+1}$, where $G_n={\rm GL}(2n-1,{\mathbb R})$. We consider here only the real matrices.

The {\it Hilbert-Lie group} ${\rm GL}_2(a)$ we  define (see \cite{Kos88}) by its
{\it Hilbert-Lie algebra} ${\mathfrak g}l_2(a)$ with composition $[x,y]=xy-yx$
$$
{\mathfrak g}l_2(a)=\{x=\sum_{k,n\in{\mathbb Z}}x_{kn}E_{kn}\mid
\Vert x \Vert^2_{{\mathfrak g}l_2(a)}= \sum_{k,n\in{\mathbb Z}}\mid
x_{kn}\mid^2a_{kn}<\infty\},\,\,a\in {\mathfrak A}_{\rm GL},
$$
$$
{\rm GL}_2(a)=\{I+x\,\mid\,(I+x)^{-1}=1+y\quad x,y\in {\mathfrak g}l_2(a)\}.
$$
\index{algebra!Hilbert-Lie}\index{group!Hilbert-Lie}
To be more precise, let us consider an analogue $\sigma_2(a)$  of the algebra of the Hilbert-Schmidt operators $\sigma_2(H)$ in a Hilbert space
$H$:
$$
\sigma_2(a)=\{x=\sum_{k,n\in{\mathbb Z}}x_{kn}E_{kn}\mid
\Vert x \Vert^2_{\sigma_2(a)}= \sum_{k,n\in{\mathbb Z}}\mid
x_{kn}\mid^2a_{kn}<\infty\}.
$$
\begin{lem}[\cite{Kos88}] The Hilbert space $\sigma_2(a)$ is an (associative) Hilbert algebra
(i.e. $\Vert xy\Vert\leq C \Vert x\Vert\Vert y\Vert,\,\,x,y\in \sigma_2(a)$) if and only if the weight
$a=(a_{kn})_{(k,n)\in {\mathbb Z}^2}$ belongs to the set ${\mathfrak A}_{\rm GL}$ defined as follows:
\begin{equation}
\label{weight(a)} {\mathfrak A}_{\rm GL}=\{a=(a_{kn})_{(k,n)\in {\mathbb
Z}^2}\mid 0< a_{kn}\leq Ca_{km}a_{mn},\quad k,n,m\in{\mathbb
Z},\,C>0\}.
\end{equation}
\end{lem}
We define the Hilbert-Lie algebra ${\mathfrak g}l_2(a)$ as the Hilbert space $\sigma_2(a)$ with an operation  $[x,y]=xy-yx$.
\begin{cor}
The Hilbert space ${\mathfrak g}l_2(a)$ is a Hilbert-Lie algebra if and only if the weight $a=(a_{kn})_{(k,n)\in {\mathbb Z}^2}$
belongs to the set ${\mathfrak A}_{\rm GL}$.
\end{cor}
We remark also \cite{Kos88} that ${\rm GL}_0(2\infty,{\mathbb R})=\cap_{a\in{\mathfrak A}_{\rm GL}}{\rm
GL}_2(a)$.
\begin{thm}[Theorem 6.1 \cite{Kos88}]
\label{ext-GL-2(a)} Every continuous unitary representation $U$ of the
group ${\rm GL}_0(2\infty,{\mathbb R})$ in a Hilbert space $H$ can be extended by
continuity to a unitary representation $U_2(a):{\rm GL}_2(a)\rightarrow
U(H)$ of some Hilbert-Lie group ${\rm GL}_2(a)$ depending on the
representation.
\end{thm}
\subsection{Hilbert-Lie groups $B_2(a)$}
\label{Hilb-Lie-B_2(a)}
Let us consider the following Hilbert-Lie group $B_2(a):=
B_2^{\mathbb Z}(a)$
\begin{equation}
\label{B_2(a),Z} B_2(a)=\{I+x\,\mid\, x\in {\mathfrak b}_2(a)\},
\end{equation}
where the corresponding Hilbert-Lie algebra ${\mathfrak
b}_2(a):={\mathfrak b}^{\mathbb Z}_2(a)$ is defined as
\begin{equation}
\label{b_2(a),Z} {\mathfrak b}_2(a)=\{ x=\sum_{(k,n)\in {\mathbb Z}^2,k<n}x_{kn}E_{kn}\,\,
\vert \,\,\Vert x\Vert^2_{{\mathfrak
b}_2(a)}=\sum_{(k,n)\in {\mathbb Z}^2,k<n}\mid x_{kn}\mid^2a_{kn}<\infty \}.
\end{equation}
\begin{lem}[\cite{Kos88}] The Hilbert space ${\mathfrak b}_2(a)$ (with an operation $(x,y)\mapsto xy$) is a Banach
algebra if and only if the weight $a=(a_{kn})_{(k,n)\in {\mathbb
Z}^2,k<n}$ satisfies the conditions
\begin{equation}
\label{b_2(a)-alg,Z} a=(a_{kn})_{k<n},\,\,a_{kn}\leq
Ca_{km}a_{mn},\,\,k<m<n,\,\,k,m,n\in{\mathbb Z}.
\end{equation}
\end{lem}
Denote by  ${\mathfrak A}$ the set of all  weight $a$ satisfying the
mentioned condition.

\subsection{Orbit method for infinite-dimensional ``nilpotent'' group  $B_0^{\mathbb Z}$, first steps}
\label{orb-meth,1-step}
Take the group $B_0^{\mathbb Z}$, fix some its Hilbert completion
i.e. a Hilbert-Lie group $B_2(a),\,\,a\in {\mathfrak A}$ and the
corresponding Hilbert-Lie algebra ${\mathfrak g}={\mathfrak
b}_2(a)$. The corresponding dual space ${\mathfrak g}^*={\mathfrak b}_2^*(a)$ has the
form
\begin{equation}
\label{b*_2(a),Z}
{\mathfrak b}_2^*(a)=\{ y=\sum_{(k,n)\in {\mathbb Z}^2,k>n}y_{kn}E_{kn}
\,\,\vert\,\, \Vert y\Vert^2_{{\mathfrak b}_2^*(a)}=\sum_{(k,n)\in {\mathbb Z}^2,k>n}\mid y_{kn}\mid^2a_{kn}^{-1}<\infty \}.
\end{equation}
The {\it adjoint action} $B_2(a)\rightarrow {\rm Aut}({\mathfrak
b}_2(a))$ of the group $B_2(a)$ on its Lie algebra ${\mathfrak
b}_2(a)$ is:
\begin{equation}
\label{Ad(3),Z} {\mathfrak b}_2(a)\ni x \mapsto {\rm
Ad}_t(x):=txt^{-1}\in {\mathfrak b}_2(a),\quad t\in B_2(a).
\end{equation}
The {\it pairing} between ${\mathfrak g}={\mathfrak b}_2(a)$ and
${\mathfrak g}^*={\mathfrak b}_2^*(a)$ is correctly defined by the
trace:
\begin{equation}
\label{<,>_2(a),Z} {\mathfrak g}^*\times {\mathfrak g}\ni
(y,x)\mapsto\langle y,x\rangle :=tr(xy)=\sum_{(k,n)\in {\mathbb Z}^2,k<n}x_{kn}y_{nk}\in{\mathbb R}.
\end{equation}
The {\it coadjoint action} of the group $B_2(a)$ on the dual
${\mathfrak g}^*={\mathfrak b}^*_2(a)$ to ${\mathfrak g}={\mathfrak
b}_2(a)$ is as follows: for $t\in B_2(x)$ and $y\in {\mathfrak
b}_2^*(a)$
$$
t=I\!+\!\!\sum_{(k,n)\in {\mathbb Z}^2,k<n}t_{kn}E_{kn},\,\,
y=\sum_{(k,n)\in {\mathbb Z}^2,k>n}y_{kn}E_{kn},\,\,t^{-1}:=
I\!+\!\!\sum_{(k,n)\in {\mathbb Z}^2,k<n}t_{kn}^{-1}E_{kn}
$$
we have
$$
(t^{-1}yt)_{pq}=\sum_{m=-\infty}^q(t^{-1}y)_{pm}t_{mq}=\sum_{m=-\infty}^q\sum_{r=p}^\infty
t^{-1}_{pr}y_{rm}t_{mq},\,\, (p,q)\in {\mathbb Z}^2,p>q,
$$
hence
\begin{equation} \label{Ad*(n),Z} {\rm
Ad}^*_x(y)=(t^{-1}yt)_-:=I+\sum_{(p,q)\in {\mathbb
Z}^2,p>q}(t^{-1}yt)_{pq}E_{pq}.
\end{equation}
We  consider four different type of orbits with respect to the
coadjoint action of the group $B_2(a)$ in the dual space
${\mathfrak b}_2^*(a)$.

{\it Case 1) The finite-dimensional orbits} corresponding to a {\it
finite points} $y=$\\$\sum_{(k,n)\in{\mathbb Z},k>n}y_{kn}E_{kn}\in
{\mathfrak b}_2^*(a)$ (finiteness of $y$ means that only finite
number of $y_{kn}$ are nonzero). This orbits leads to the induced
representations of an appropriate finite-dimensional groups
$G_n^{m},\,\,m\in{\mathbb Z},\,\,n\in {\mathbb N}$  defined by
(\ref{G^m_n=}). All irreducible unitary representations  of the
groups $G_n^{m}$ are completely described by the Kirillov orbit
method hence the finite-dimensional orbits gives us the set $\bigcup_{n\in{\mathbb
N}}\widehat{G^{m}_n}\subset\widehat{B_0^{\mathbb Z}} $ (see
subsection~\ref{dual-astep}, Remark~\ref{orbit-fin} for embedding
$\widehat{G^{m}_n}\subset\widehat{G^{m}_{n+1}}$).

{\it Case 2) $0$-dimensional orbits} are of the form:
$$
{\mathcal O}_0=y,\,\, y\in {\mathfrak b}_2^*(a),\quad
y=\sum_{k\in{\mathbb Z}} y_{k+1,k}E_{k+1,k}.
$$
The Lie algebra ${\mathfrak b}_2(a)$ is subordinate to the
functional $y$,  $\langle y,[{\mathfrak b}_2(a),{\mathfrak b}_2(a)]\rangle=0$ since
$$
[{\mathfrak b}_2(a),{\mathfrak b}_2(a)]=\{x\in {\mathfrak b}_2(a)\mid x=\sum_{(k,n)\in {\mathbb Z}^2,k+1<n}x_{kn}E_{kn}\}.
$$
 The one-dimensional representation of
the Lie algebra ${\mathfrak b}_2(a)$ are
$$
{\mathfrak b}_2(a)\ni x\mapsto \langle y,x\rangle=\sum_{k\in{\mathbb Z}} x_{k,k+1}y_{k+1,k}\in{\mathbb R}.
$$
Corresponding one-dimensional representations of the group $B_2(a)$
are as follows:
\begin{equation}
\label{1-dim,rep,Z} B_2(a)\ni \exp(x)\mapsto \exp (2\pi i(\langle
y,x\rangle))= \exp (2\pi i\sum_{k\in{\mathbb Z}}
x_{k,k+1}y_{k+1,k})\in S^1.
\end{equation}
They are all {\it irreducible and nonequivalent} for different $y=\sum_{k\in{\mathbb Z}} y_{k+1,k}E_{k+1,k}\in {\mathfrak b}_2^*(a).$

{\it Case 3) Generic orbit} is generated for an arbitrary $m\in
{\mathbb Z}$ by a point $y\in{\mathfrak b}_2^*(a)$
\begin{equation}
\label{y-in-g*} y\!=\!\sum_{p=0}^\infty
y_{m+p+1,m-p}E_{m+p+1,m-p}\in {\mathfrak b}_2^*(a),\,\,\text{
with}\,\,y_{m+p+1,m-p}\not=0,\,\,p+1\in{\mathbb N}.
\end{equation}
Sections~\ref{gen-orb,Z} and \ref{gen-orb,Z,irr} are devoted to the
study of this case.

{\it Case 4) General orbits} generated by an arbitrary non finite points
$$y=\sum_{(k,n)\in{\mathbb Z},k>n}y_{kn}E_{kn}\in{\mathfrak b}_2^*(a).$$
{\bf Problem}. How to construct the induced representations for  general orbits and
study their irreducibility?

\subsection{Construction of the induced representations of the group $B_0^{\mathbb Z}$ corresponding to a generic orbits}
\label{gen-orb,Z}
Consider more carefully the case 3). The irreducibility we shall study in the following subsection.
Take as before the group $B_0^{\mathbb Z}$, fix  some its Hilbert completion
i.e. a Hilbert-Lie group $B_2(a),\,\,a\in {\mathfrak A}$, the
corresponding Hilbert-Lie algebra ${\mathfrak g}={\mathfrak
b}_2(a)$ and its dual ${\mathfrak g}^*={\mathfrak
b}_2^*(a)$ as in  the previous subsection.

We shall write the analog of the induced representation of the group
$B_0^{\mathbb Z}$ for generic orbits (see Examples
\ref{Gen-orbit.n}, \ref{B_5-orb} and \ref{B_6-orb}) corresponding to
the point  $y\in{\mathfrak b}_2^*(a)$ defined by (\ref{y-in-g*})
following steps 1)--3) of Definition \ref{df.ind-inf-dim}.

Step 1) {\it Extension of the representation} $S:H\rightarrow U(V)$.
For fixed $m\in{\mathbb Z}$, consider the decomposition
$$B^{\mathbb Z}=B_mB(m)B^{(m)}$$
 similar to the decomposition (\ref{2decompos}),
where $B^{\mathbb Z}=\{I+\sum_{k,n\in{\mathbb
Z},\,k<n}x_{kn}E_{kn}\}$,
$$
B_m=\{I+\sum_{(k,r)\in\Delta_{m}}x_{kr}E_{kr}\},\,\,
B(m)=\{I+\sum_{(k,r)\in\Delta(m)}x_{kr}E_{kr}\},\,\,
B^{(m)}=\{I+\sum_{(k,r)\in\Delta^{(m)}}x_{kr}E_{kr}\},
$$
$$\Delta_{m}=\{(k,r)\in{\mathbb Z}^2\mid m+1\leq k<r\},\quad
\Delta(m)=\{(k,r)\in{\mathbb Z}^2\mid k\leq m<r\},\,\,$$
and $\Delta^{(m)}=\{(k,r)\in{\mathbb Z}^2\mid k<r\leq m\}.$\\
Since the algebras ${\mathfrak h}_{0}(m),\,\,m\in {\mathbb Z}$
defined as follows ${\mathfrak h}_{0}(m)=\{t-I\mid t\in B_0(m)\}$,
where $B_0(m)=B(m)\cap B_0^{\mathbb Z}$, are commutative, so
$\langle y,[{\mathfrak h}_{0}(m),{\mathfrak h}_{0}(m)]\rangle=0$,
hence they are subordinate to the functional $y\in {\mathfrak
g}^*={\mathfrak b}_2^*(a)$. The corresponding one-dimensional
representation of the algebra ${\mathfrak h}_{0}(m)={\mathfrak
h}(m)\bigcap {\mathfrak g}_0^{\mathbb Z}$ is
$$
{\mathfrak h}_{0}(m)\ni x\mapsto \langle y,x\rangle=\sum_{p=0}^\infty x_{m-p,m+p+1}y_{m+p+1,m-p}\in {\mathbb R}.
$$
The unitary representation of the corresponding group  $H_0(m)$ is
$$
H_0(m)\ni \exp(x)\mapsto S(\exp(x))=\exp(2\pi i\langle y,x\rangle)\in S^1.
$$
This representation can be extended to representation of the
corresponding Hilbert-Lie group ${\tilde H}=H_{2}(m,a)=B(m)\bigcap
B_2(a)$ (we note that $t=\exp(t-1)$):
$$
H_{2}(m,a)\ni \exp(x)\mapsto S(\exp(x))=\exp(2\pi i\langle y,x\rangle)\in S^1.
$$
In what follows we shall use a notation $B_{2}(m,a)$ for the group
$H_{2}(m,a)$.

Step 2 a) {\it Construction of the completion} $\tilde X=\tilde
H\backslash \tilde G$ of the space $X=H\backslash G$. It is
difficult to construct an appropriate measure on the space $X_{m,0}=B_0(m)\backslash B_0^{\mathbb Z}$ since it is
isomorphic to the space ${\mathbb R}_0^\infty\subset {\mathbb R}_0^\infty$. That is why we  consider two
homogeneous spaces, an appropriate completions of the space $X_{m,0}$:
$$
X_{m,2}(a)=B_{m,2}(a)\backslash B_2(a),\quad X_{m}=B(m) \backslash
B^{\mathbb Z}.
$$
Since the decompositions holds
$$
B_0^{\mathbb Z}=B_{m,0}B_0(m)B_0^{(m)},\quad
B_2(a)=B_{m,2}(a)B_{2}(m,a) B^{(m)}_2(a),\quad
B^{\mathbb Z}=B_mB(m)B^{(m)},
$$
(see Remark \ref{9.B(n,R)=LAU}), we have the following
inclusions: $ X_{m,0}\subset X_{m,2}(a)\subset X_{m}$, where
$$
X_{m,0}\simeq B_{m,0}\times B^{(m)}_0,\quad
X_{m,2}(a)\simeq B_{m,2}(a)\times B^{(m)}_2(a),\quad
X_{m}=B(m)\backslash B^{\mathbb Z}\simeq B_m\times B^{(m)}.
$$
Step 2 b) We {\it construct a measure} $\mu_b$ on the space $X_{m}$
with support $X_{m,2}(a)$ i.e. such that
$\mu_b(X_{m,2}(a))=1$.
 That is we take $\tilde X=\tilde H\backslash \tilde G=B_{2}(m,a)\backslash B_2(a)$.
\begin{rem}
\label{mes-on-X} On the space $X_m$ we can take any
$B_0^{\mathbb Z}$-quasi-invariant ergodic measure, construct the
induced representation and study the irreducibility. We consider the simplest case of the Gaussian
measure, the infinite product of one-dimensional Gaussian measure.
\end{rem}
We construct the measure $\mu_b$ on the space $X_{m}\simeq B_m\times B^{(m)}$ as a product-measure $\mu_b=\mu_{b,m}\otimes \mu_b^{(m)}$, where
$\mu_{b,m}$ (resp. $\otimes \mu_b^{(m)}$) is Gaussian product measure on the group $B_m$ (resp. $B^{(m)}$) defined as follows:
\begin{eqnarray}
\label{Gauss-on-G^m_n} \,\,\,d\mu_{b,m}(x_m)\!=\!\otimes_{(k,n)\in
\Delta_{m}}d\mu_{b_{kn}}(x_{kn})\!=\!\otimes_{(k,n)\in
\Delta_{m}}\sqrt{\frac{b_{kn}}{\pi}}
\exp(-b_{kn}x_{kn}^2)dx_{kn},\\
\,\,\,\,\,\,d\mu_b^{(m)}(x^{(m)})\!=\!\otimes_{(k,n)\in
\Delta^{(m)}}d\mu_{b_{kn}}(x_{kn})\!=\!\otimes_{(k,n)\in
\Delta^{(m)}}\sqrt{\frac{b_{kn}}{\pi}} \exp(-b_{kn}x_{kn}^2)dx_{kn}.
\end{eqnarray}
The corresponding Hilbert space is
$${\mathcal H}^m=L^2(X_m,\mu_b)=L^2(B_m\times
B^{(m)},\mu_{b,m}\otimes \mu_b^{(m)}).
$$
\begin{lem}
[Kolmogorov's zero-one law, \cite{ShFDT67}]
We have $\mu_{b,m}\otimes \mu_b^{(m)}(B_{m,2}(a)\times B^{(m)}_2(a))=1$ if and only if
$$\sum_{(k,n)\in \Delta(m)\cup \Delta^{(m)}}
\frac{a_{kn}}{b_{kn}}<\infty.$$
\end{lem}
\begin{lem}[\cite{Kos90,Kos92}]
\label{mu-b-quasi-inv} The measure $\mu_b=\mu_{b,m}\otimes
\mu_b^{(m)}$ is $B_{m,0}\times B^{(m)}_0$-right-quasi-invariant
i.e. $(\mu_b)^{R_t}\sim \mu_b$ for all $t\in B_{m,0}\times B^{(m)}_0$
if and only if
$$
S^{R}_{kn}(\mu_{b})=\sum_{r=-\infty}^{k-1}\frac{b_{rn}}{b_{rk}}<\infty,
\quad\text{for all},\,\,k<n\leq m.
$$
\end{lem}
Step 3) The corresponding induced representation of the group $B_0^{\mathbb Z}$ we defined as
follows:
\begin{equation}
\label{Ind(B^Z_0)}
(T_t^{m,y}f)(x)=S(h(x,t))\left(\frac{d\mu_b(xt)}{d\mu_b(x)}\right)^{1/2}f(xt),\,\,x\in
X_m,\,\,t\in G,
\end{equation}
where (see (\ref{9.S(h)=exp(2pi tr(B)),Z}))
$$S(h(x,t))=\exp(2\pi i\langle y,h(x,t)-1\rangle)=
\exp\Big(2\pi i{\rm tr}\left((t-I)B(x,y)\right) \Big).
$$
\subsection{Irreducibility of the induced representations of the group $B_0^{\mathbb Z}$
corresponding to a generic orbits} \label{gen-orb,Z,irr}

Consider the induced representation $T^{m,y}$ of the group $B_0^{\mathbb Z}$
corresponding to a generic orbit ${\mathcal O}_y$, generated by the
point \\$y=\sum_{r=0}^\infty y_{m+r+1,m-r}E_{m+r+1,m-r}\in
{\mathfrak b}_2^*(a)$  defined  by (\ref{Ind(B^Z_0)}).
Set for  $(k,r)\in \Delta(m)$
\begin{equation}
\label{S_kr(t),Z}
 S_{kr}(t_{kr}):=\langle
y,(h(x,E_{kr}(t_{kr}))-I)\rangle,\quad\text{then}\quad
A_{kr}=\frac{d}{dt}\exp(2\pi i S_{kr}(t))\vert_{t=0}=2\pi i S_{kr}(1).
\end{equation}
Let us denote by ${\mathbb S}^{(m)}={\mathbb S}$ the following
matrix (compare with (\ref{S_kr(t)}) and (\ref{def:S=(S_kr)})):
\begin{equation}
\label{def:S=(S_kr),Z}
{\mathbb S}=(S_{kr})_{(k,r)\in \Delta(m)},\quad \text{where}\quad S_{kr}=S_{kr}(1).
\end{equation}

We calculate now the matrix  ${\mathbb S}(t)=(S_{kr}(t_{kr}))_{(k,r)\in \Delta(m)}$ and the matrix ${\mathbb S}=$\\
$(S_{kr}(1))_{(k,r)\in \Delta(m)}$ using analog of the Lemma~\ref{l.tr(E_knB)=B^t}. As in  (\ref{H(x,t)}) we have
$$
\langle y,h(x,t)-I\rangle={\rm tr}\left(H(x,t)y\right)={\rm tr}\left(x^{(m)}t_0x_m^{-1}y\right)=
{\rm tr}\left(t_0x_m^{-1}yx^{(m)}\right)={\rm tr}\left(t_0B(x,y)\right),
$$
where $t_0=t-I$ and  for $x_m\in B_m,\,\,x^{(m)}\in B^{(m)}$ we denote
\begin{equation}
\label{B(x,y),Z}
B(x,y)=x_m^{-1}yx^{(m)}\cong
\left(\begin{smallmatrix}
1&0\\
0&x_m^{-1}\\
\end{smallmatrix}\right)
\left(\begin{smallmatrix}
0&0\\
y&0\\
\end{smallmatrix}\right)
\left(\begin{smallmatrix}
x^{(m)}&0\\
0&1\\
\end{smallmatrix}\right)=
\left(\begin{smallmatrix}
0&0\\
x_m^{-1}yx^{(m)}&0\\
\end{smallmatrix}\right).
\end{equation}
By definition we have (recall that $E_{kn}(t_{kn})=I+t_{kn}E_{kn}$)
$$
S_{kn}(t_{kn})=\langle y,(h(x,E_{kn}(t_{kn}))-I)\rangle={\rm tr}(t_{kn}E_{kn}B(x,y)),
$$
hence by analog of the Lemma~\ref{l.tr(E_knB)=B^t} we conclude that
\begin{equation}
\label{S=B^T,Z}
{\mathbb S}=(S_{kn}(1))_{k,r}=\left({\rm tr}\left(E_{kr}B(x,y)\right)\right)_{k,r}=B^T(x,y)=
(x^{(m)})^Ty^T(x_m^{-1})^T=
\left(\begin{smallmatrix}
0&(x^{(m)})^Ty^T(x_m^{-1})^T\\
0&0\\
\end{smallmatrix}\right).
\end{equation}
So, we have
\begin{equation}
\label{9.S(h)=exp(2pi tr(B)),Z} S(h(x,t))=\exp( 2\pi i\langle
y,(h(x,t)-I)\rangle)=\exp\Big(2\pi i {\rm
tr}\left((t-I)B(x,y)\right) \Big).
\end{equation}
Using results of \cite{Kos00f}
we conclude that the following lemma holds.
\begin{lem}
\label{mu-b-ergodic}The measure $\mu_b=\mu_{b,m}\otimes \mu_b^{(m)}$
is $B_{m,0}\times B^{(m)}_0$-right-ergodic if
$$
E(\mu_{b})=\sum_{k<n\leq
m}\frac{S^{R}_{kn}(\mu_{b})}{b_{kn}}<\infty.
$$
\end{lem}
\begin{thm}
\label{Ind-infin-irr} The induced representation
$T^{m,y}$ of the group $B_0^{\mathbb Z}$ defined by
formula (\ref{Ind(B^Z_0)}), corresponding to generic orbit
${\mathcal
O}_y$, generated by the point \\
$y=\sum_{r=0}^\infty y_{m+r+1,m-r}E_{m+r+1,m-r}\in {\mathfrak b}_2^*(a)$
is irreducible if the measure $\mu_{b,m}\otimes \mu_b^{(m)}$ on the
group $B_{m}\times B^{(m)}$  is right $B_{m,0}\times
B^{(m)}_0$-ergodic. Moreover the generators of one-parameter groups
$A_{kr}=\frac{d}{dt}T^{m,y}_{I+tE_{kr}}\mid_{t=0}$ are as follows
$$
A_{kr}=\sum_{s=-\infty}^{k-1}x_{ks}D_{rs}+D_{kr},\,\, (k,r)\in \Delta^{(m)},\quad
A_{kr}=\sum_{s=m+1}^{k-1}x_{ks}D_{rs}+D_{kr},\,\, (k,r)\in \Delta_{m},
$$
$$
(2\pi i)^{-1}\big(A_{kr}\big)_{(k,r)\in \Delta(m)}={\mathbb S}^{(m)}=(S_{kr})_{(k,r)\in\Delta(m)}=\big(x_m^{-1}yx^{(m)}\big)^T.
$$
\end{thm}
Here we denote by $D_{kn}=D_{kn}(\mu_b)$ the operator of the
partial derivative corresponding to the shift $x\mapsto x+tE_{kn}$
and the measure $\mu_b$ on the group $B_m\times B^{(m)}\ni
x=I+\sum x_{kr}E_{kr}$:
\begin{equation}
\label{D_{kn}(mu)}
(D_{kn}(\mu_b)f)(x)=\frac{d}{dt}\left(\frac{d\mu_b(x+tE_{kn})}{d\mu_b(x)}\right)^{1/2}\!\!f(x+tE_{kn})\mid_{t=0},
\quad
D_{kn}(\mu_b)=\frac{\partial}{\partial x_{kn}}-b_{kn}x_{kn}.
\end{equation}
{\it The irreducibility} of the induced representation of the group $B_0^{\mathbb Z}$  follows from the following  lemma.
\begin{lem}
\label{W(S_kn)=W(x_kn)B^Z} Two von Neumann algebra ${\mathfrak
A}^S$ and ${\mathfrak A}^x$  in the space ${\mathcal
H}^{m}=L^2(X_m,\mu_b)$ generated respectively by  the sets of
unitary operators $U_{kr}(t)$ and $V_{kr}(t)$ coincides, where
\begin{equation}
\label{U_kr(t),V_{kr}(t)-(m)}
(U_{kr}(t)f)(x)=\exp(2\pi i S_{kr}(t))f(x),\quad (V_{kr}(t)f)(x):=\exp(2\pi i tx_{kr})f(x),
\end{equation}
$$
{\mathfrak A}^S=\big(U_{kr}(t)=T_{I+tE_{kr}}^{m,y}=\exp(2\pi i S_{kr}(t))\mid  t\in {\mathbb R},\,\,
(k,r)\in \Delta(m)\big)'',
$$
$$
{\mathfrak A}^x=\big(V_{kr}(t)=\exp(2\pi i tx_{kr})\mid  t\in{\mathbb R},\,\,(k,r)\in
\Delta_{m}\bigcup\Delta^{(m)}\big)''.
$$
\end{lem}
\begin{proof}
Using the decomposition (\ref{S=B^T,Z})
$$
{\mathbb S}^{(m)}=B(x,y)^T=(x_m^{-1}yx^{(m)})^T=(x^{(m)})^Ty^T(x_m^{-1})^T
$$
we conclude that ${\mathfrak A}^S\subseteq {\mathfrak A}^x$ (see
the proof of Lemma~\ref{W(S_kn)=W(x_kn)m-n}).

To prove that ${\mathfrak A}^S\supseteq {\mathfrak A}^x$ it is
sufficient  to find the expressions of the matrix element of the
matrix $x^{(m)}\in B^{(m)}$ and $x_m^{-1}\in B_m$ in terms of the
matrix elements of the matrix ${\mathbb
S}^{(m)}=(S_{kr})_{(k,r)\in \Delta(m)}$.
To do this we connect the above decomposition ${\mathbb
S}^{(m)}=B(x,y)^T$ (see (\ref{B(x,y),Z})) and the Gauss
decomposition $C=LDU$ for infinite matrices  (see
Theorem~\ref{t.C=LDU,Z1}). By (\ref{B(x,y),Z}) we get
$B(x,y)=x_m^{-1}yx^{(m)}$.

To find a matrix connected with the matrix ${\mathbb S}^{(m)}$, for which an appropriate decomposition
$LDU$ holds we {\it recall the expressions for  $B(x,y)$  for small} $n$ and finite-dimensional groups $G_n^m$ (see Example (\ref{B_6-orb})).
We note that $J_m^2=I$, where
$$
J_m\in{\rm Mat}(\infty,{\mathbb R}),\quad J_m=\sum_{r\in{\mathbb Z}}E_{m+r+1,m-r}.
$$
For $G^3_3$ we get
$$
B(x,y)=x_m^{-1}yx^{(m)}= \left(\begin{smallmatrix}
1&x_{45}^{-1}&x_{46}^{-1}&x_{47}^{-1}\\
0&1&x_{56}^{-1}&x_{57}^{-1}\\
0&0&1&x_{67}^{-1}\\
0&0&0&1\\
\end{smallmatrix}\right)
\left(\begin{smallmatrix}
0&0&0&y_{43}  \\
0&0&y_{52}  &0\\
0&y_{61}&0  &0\\
y_{70}&0  &0&0\\
\end{smallmatrix}\right)
\left(\begin{smallmatrix}
1&x_{01}&x_{02}&x_{03}\\
0&1     &x_{12}&x_{13}\\
0&0&1&x_{23}  \\
0&0&0&1\\
\end{smallmatrix}\right),
$$
\begin{equation}
\label{C_8} B(x,y)J= \left(\begin{smallmatrix}
1&x_{45}^{-1}&x_{46}^{-1}&x_{47}^{-1}\\
0&1&x_{56}^{-1}&x_{57}^{-1}\\
0&0&1&x_{67}^{-1}\\
0&0&0&1\\
\end{smallmatrix}\right)
\left(\begin{smallmatrix}
y_{43}&0&0&0  \\
0&y_{52}&0  &0\\
0&0&y_{61}  &0\\
0&0  &0&y_{70}\\
\end{smallmatrix}\right)
\left(\begin{smallmatrix}
1&0&0&0\\
x_{23}&1&0&0  \\
x_{13}&x_{12}&1& 0\\
x_{03}&x_{02}& x_{01}&1\\
\end{smallmatrix}\right).
\end{equation}
We use the infinite-dimensional analog of the latter presentation, i.e. instead of the group $G_n=B(n,{\mathbb R})$ consider
the infinite-dimensional group $B_0^{\mathbb Z}$ and do the same. Let
$$x_m\in B_m,\,\,x^{(m)}\in B^{(m)},\,\,y=\sum_{r=0}^\infty
y_{m+r+1,m-r}E_{m+r+1,m-r}\in {\mathfrak g}^*_2(a)
$$
and $J=J_m= \sum_{r\in {\mathbb Z}}E_{m+r+1,m-r}$. Then we get
${\mathbb S}^T=B(x,y)=x_m^{-1}yx^{(m)}.$

Set $C=C(x,y)=B(x,y)J$ then $C=UDL$, more precisely we have:
\begin{equation}
\label{BJ=UDL}
B(x,y)J=x_m^{-1}yJ_mJ_mx^{(m)}J_m=UDL,\,\,\text{where}\,\,
U=x_m^{-1},\,\,D=yJ_m,\,\,L=J_mx^{(m)}J_m,
\end{equation}
\begin{equation}
\label{C=BJ=UDL}
C=B(x,y)J=
\left(\begin{smallmatrix}
1&x_{45}^{-1}&x_{46}^{-1}&x_{47}^{-1}&\dots\\
0&1&x_{56}^{-1}&x_{57}^{-1}&\dots\\
0&0&1&x_{67}^{-1}&\dots\\
0&0&0&1&\dots\\
 &&&&\dots\\
\end{smallmatrix}\right)
\left(\begin{smallmatrix}
y_{43}&0&0&0  &\dots\\
0&y_{52}&0  &0&\dots\\
0&0&y_{61}  &0&\dots\\
0&0  &0&y_{70}&\dots\\
 &&&&\dots\\
\end{smallmatrix}\right)
\left(\begin{smallmatrix}
1&0&0&0&\dots\\
x_{23}&1&0&0 &\dots \\
x_{13}&x_{12}&1& 0&\dots\\
x_{03}&x_{02}& x_{01}&1&\dots\\
 &&&&\dots\\
\end{smallmatrix}\right),
\end{equation}
\begin{equation*}
C=\left(\begin{smallmatrix}
c_{11}&c_{12}&\dots&c_{1n}&\dots\\
c_{21}&c_{22}&\dots&c_{2n}&\dots\\
            &&\dots&&\dots\\
c_{n1}&c_{n2}&\dots&c_{nn}&\dots\\
 &&\dots&&\dots\\
 \end{smallmatrix}\right)=
 \left(\begin{smallmatrix}
1&u_{12}&\dots&u_{1n}&\dots\\
0&1&\dots&u_{2n}&\dots\\
            &&\dots&&\dots\\
0&0&\dots&1&\dots\\
 &&\dots&&\dots\\
 \end{smallmatrix}\right)
\left(\begin{smallmatrix}
d_{1}&0&\dots&0&\dots\\
0&d_{2}&\dots&0&\dots\\
            &&\dots&&\dots\\
0&0&\dots&d_{n}&\dots\\
 &&\dots&&\dots\\
\end{smallmatrix}\right)
\left(\begin{smallmatrix}
1&0&\dots&0\\
l_{21}&1&\dots&0&\dots\\
            &&\dots&&\dots\\
l_{n1}&l_{n2}&\dots&1&\dots\\
 &&\dots&&\dots\\
 \end{smallmatrix}\right).
 \end{equation*}

 To finish the proof of the Lemma it is sufficient to find the
decomposition (\ref{C=BJ=UDL}) $C=UDL$ .

Let us {\it suppose that we can find the inverse
matrix} $C^{-1}$. Then by (\ref{BJ=UDL}) holds
$C^{-1}=L^{-1}D^{-1}U^{-1}$ and we can use
Theorem~\ref{t.C=LDU,Z1} to find
$$L^{-1}=J_m(x^{(m)})^{-1}J_m,\quad D^{-1}=y^{-1}J_m,\quad U^{-1}=x_m.
$$
Hence, we can find the matrix elements of
the matrix $(x^{(m)})^{-1}\in B^{(m)}$ and $x_m\in B_m$
in terms of the matrix elements of the matrix $C^{-1}=({\mathbb S}^TJ)^{-1}=(B(x,y)J)^{-1}$.
Finally, we can also find the matrix elements of the matrix $x^{(m)}\in B^{(m)}$ using formulas
(\ref{x^{-1}_{kn}:=}).  This finish the proof of the lemma since in this case we have
$x_{kr}\,\,\eta\,\, {\mathfrak A}^S$ for  $(k,r)\in\Delta_{m}\bigcup\Delta^{(m)}$. Hence
${\mathfrak A}^S\subseteq{\mathfrak A}^x$.

1) To find the inverse matrix $C^{-1}$ we write two decompositions:
\begin{equation}
\label{C=LUD=U'L'D'}
C=L_1D_1U_1=UDL,\quad C^{-1}=(U_1)^{-1}(D_1)^{-1}(L_1)^{-1}=L^{-1}D^{-1}U^{-1}.
\end{equation}
2) Using (\ref{C=LUD=U'L'D'}) we can find $L_1,D_1$ and $U_1$ by  Theorem~\ref{t.C=LDU,Z1}. More precisely,
for  all $x\in \Gamma_G$, where
$$
\Gamma_C=\{x\in B_m\times B^{(m)}\mid M^{12\dots k}_{12\dots k}(C(x))\not=0,\,\,k\in{\mathbb N}\}
$$
holds the decomposition $C(x)=L_1D_1U_1$ and the matrix elements of the matrix $L_1,\,\,D_1$ and $U_1$ are rational functions in $c_{kn}(x)$. \\
3) We can find $(L_1)^{-1}$ and $(U_1)^{-1}$ using  formulas (\ref{x^{-1}_{kn}:=}). Note that
$J_mLJ_m,\,U,\,$ and $J_mL^{-1}J_m,\,U^{-1}\in B_2(a)$.\\
4) Using identity (\ref{C=LUD=U'L'D'})  we can calculate  $C^{-1}=(U_1)^{-1}(D_1)^{-1}(L_1)^{-1}$, since
$L^{-1},\,\,D^{-1}$ and $U^{-1}$ are well defined.  \\
5) Using equality (\ref{C=LUD=U'L'D'}) we can find the decomposition $C^{-1}=L^{-1}D^{-1}U^{-1}$ of the matrix
$C^{-1}$ by Theorem~\ref{t.C=LDU,Z1}. In other words, the decompositions holds $C^{-1}=L^{-1}D^{-1}U^{-1}$ for all
$x\in \Gamma_{G^{-1}}$, where
$$
\Gamma_{C^{-1}}=\{x\in B_m\times B^{(m)}\mid M^{12\dots k}_{12\dots k}(C^{-1}(x))\not=0,\,\,k\in{\mathbb N}\}
$$
and the matrix elements of the matrix $L^{-1},\,\,D^{-1}$ and $U^{-1}$ are rational
functions in matrix elements $c^{-1}_{kn}(x)$ of the matrix $C^{-1}$.

We make the last remark.
Let us denote $(L_1)^{-1}=(L_{1;kn}^{-1})_{kn},(D_1)^{-1}={\rm diag}(d^{-1}_{1;k})_k$ and $(U_1)^{-1}=(U_{1;kn}^{-1})_{kn}$.
The decompositions $C=L_1D_1U_1$ and $C^{-1}=(U_1)^{-1}(D_1)^{-1}\times$ $(L_1)^{-1}$ hold for $x\in \Gamma_C\cap\Gamma_{C^{-1}}$, i.e. almost for all
$x\in B_m\times B^{(m)}$ with respect to the measure $\mu_b$ since $\mu_b(\Gamma_C\cap\Gamma_{C^{-1}})=1$. We conclude that the convergence
$$
c_{kn}^{-1}(x)=\sum_{m\in \mathbb N}U_{1;km}^{-1}d^{-1}_{1;m}L_{1;mn}^{-1},\,\,k,n\in {\mathbb N}
$$
holds  pointwise almost everywhere $x\in B_m\times B^{(m)}\,\,({\rm
mod}\,\,\mu_b)$. Since $U_{1;km}^{-1},\,\,d^{-1}_{1;m}$ and
$L_{1;mn}^{-1}\,\eta\,{\mathfrak A}^S$ by 2) and 3), we conclude by Lemma~\ref{lim-f(n)=f-in-von-Neu}  that
$c_{kn}^{-1}(x)\,\eta\,{\mathfrak A}^S$. This finish the proof of
the lemma. 
\end{proof}
\begin{proof} {\it of the Theorem~\ref{Ind-infin-irr}}. To prove the irreducibility of the induced representation
consider the restriction $T^{m,y}\mid_{B_0(m)}$ of this representation to the commutative subgroup $B_0(m)$ of
the group $B_0^{\mathbb Z}$. Note that
$$
{\mathfrak A}^x=\big(\exp(2\pi i tx_{kr})\mid t\in {\mathbb R},\,\,  (k,r)\in \Delta_{m}\bigcup\Delta^{(m)}
\big)''=L^\infty(B_m\times B^{(m)},\mu_{b,m}\otimes \mu_b^{(m)}).
$$
By Lemma~\ref{W(S_kn)=W(x_kn)B^Z} the von Neumann algebra ${\mathfrak A}^S$ generated by this restriction
coincides with ${\mathfrak A}^x=L^\infty(B_m\times B^{(m)},\mu_{b,m}\otimes \mu_b^{(m)})$. Let now a bounded operator $A$ in
the Hilbert space ${\mathcal H}^{m}$
commute with the representation $T^{m,y}$. Then $A$ commute by the above arguments with $L^\infty(B_m\times
B^{(m)},\mu_{b,m}\otimes \mu_b^{(m)})$, therefore the operator $A$ itself is an operator of multiplication by some essentially bounded function
 $a\in L^\infty$ i.e. $(Af)(x)=a(x)f(x)$ for
$f\in {\mathcal H}^{m}$. Since $A$ commute with the representation $T^{m,y}$ i.e. $[A,T^{m,y}_t]=0$ for all
$t\in B_{m,0}\times B^{(m)}_0$, where $B_{m,0}=B_m\cap B_0^{\mathbb Z}$ and
$B_0^{(m)}=B^{(m)}\cap B_0^{\mathbb Z}$, we conclude  that
$$
a(x)=a(xt)\,\,({\rm mod}\,\, \mu_{b,m}\otimes \mu_b^{(m)})\quad \text{for \quad all}\quad t\in B_{m,0}\times B^{(m)}_0.
$$
Since the measure $\mu_{b,m}\otimes \mu_b^{(m)}$ on the group $B_{m}\times B^{(m)}$
is right $B_{m,0}\times B^{(m)}_0$-ergodic we conclude
that $a(x)=const$ $({\rm mod}\,\,dx_m\otimes dx^{(m)})$.
\end{proof}
\begin{rem} We would like to show that $T^{m,y}=\lim_nT^{m,y_n}$. To be more precise consider the projection
$B^{\mathbb Z}_0\mapsto G^{m}_n$ of the group $B^{\mathbb Z}_0$ on the subgroup $G^{m}_n$ and all other projections: homogeneous spaces,
 measures, Hilbert spaces and representations:
$$
X_m=B_m\times B^{(m)}\mapsto X_{m,n}=B_{m,n}\times B^{(m,n)},\quad \mu_{b,m}\otimes \mu_b^{(m)}\mapsto \mu_{b,m,n}\otimes \mu_b^{(m,n)}
$$
$$
{\mathcal H}^{m}=L^2(B_m\times B^{(m)},\mu_{b,m}\otimes \mu_b^{(m)})\mapsto L^2(B_{m,n}\times B^{(m,n)},
\mu_{b,m,n}\otimes \mu_b^{(m,n)})
$$
$$
\cong L^2(B_{m,n}\times B^{(m,n)},dx_{m,n}\otimes dx^{(m,n)})={\mathcal H}^{m,n}
$$
$$
T^{m,y}\mapsto T^{m,y_n},\quad n\in{\mathbb N}.
$$
\end{rem}
Since the measure $\mu_{b,m,n}\otimes \mu_b^{(m,n)}$ is equivalent with the Haar measure
(compare (\ref{Haar-on-G^m_n}) and (\ref{Gauss-on-G^m_n})) we conclude that the corresponding representations
$T^{\mu,m,y_n}$ in the spaces $L^2(B_{m,n}\times B^{(m,n)},
\mu_{b,m,n}\otimes \mu_b^{(m,n)})$ and $T^{m,y_n}$ in the space
$L^2(B_{m,n}\times B^{(m,n)},dx_{m,n}\otimes dx^{(m,n)})$ are equivalent.
This implies $T^{m,y}=\lim_nT^{m,y_n}$.
\subsection{Dual description of the groups $B_0^{\mathbb N}$ and $B_0^{\mathbb Z}$. First steps.}
\label{dual-astep}
Let $\hat{G}$ be the dual of the group $G$. Our {\it aim} is to
describe $\hat{G}$ for $G=\varinjlim_{n}G_n$ where
$G_n=B(n,{\mathbb R})$ is the group of all $n\times n$ upper
triangular real matrices with units on the principal diagonal,
i.e. we would like to describe the dual of the group $B_0^{\mathbb
N}$ of infinite in one direction and $B_0^{\mathbb Z}$ infinite in both directions matrices.
Consider the  inductive limit $G=\varinjlim_{n}G_n$ of nilpotent
groups $G_n=B(n,{\mathbb R})$. The symmetric (resp. nonsymmetric)
imbedding gives us two infinite-dimensional analog of ``nilpotent''
groups  $B_0^{\mathbb Z}$ (resp. $B_0^{\mathbb N}$).

{\it We do not know the description of all} $\hat{G}$. We only know that the set $\hat{G}$ contains the following
three classes of representations. \\
1) The set $\hat{G}$ contains $\bigcup_{n}\hat{G_n}$ i.e. $\hat{G}\supset\bigcup_{n}\hat{G_n}$.
One may use Kirillov's
orbit method \cite{Kir62, Kir94} to describe $\hat{G_n}$. The
embedding $\hat{G_n}\subset\hat{G_{n+1}}$ is described in
Remark~\ref{orbit-fin}.\\
2) We have $\hat{G}\setminus\bigcup_{n}\hat{G_n}\not=\emptyset.$ Namely
$\hat{G}\setminus\bigcup_{n}\hat{G_n}$ contains "regular"
$T^{R,\mu}$ and "quasiregular" $\pi^{R,\mu,X}$ representations of
the group $G$ (see subsection~\ref{reg+quasireg-rep}).\\
3) Induced representations (see subsection~\ref{orb-meth,1-step}).

It is natural together with the group $B_0^{\mathbb N}$ (resp. $B_0^{\mathbb Z}$) consider
all Hilbert-Lie completion $B^{\mathbb N}_2(a)$ (resp. $B^{\mathbb Z}_2(a)$) and the group of all
upper-triangular matrices $B^{\mathbb N}$ (resp. $B^{\mathbb Z}$)
(see subsections \ref{Hilb-Lie-B_2(a)}, \ref{Hilb-Lie-GL_2(a)})
\begin{equation*}\label{all-Group1}
G_n\rightarrow B_0^{\mathbb N}\rightarrow B^{\mathbb N}_2(a)\rightarrow
B^{\mathbb N}\rightarrow G_n.
\end{equation*}
\begin{equation*}
G_n^m\rightarrow B_0^{\mathbb Z}\rightarrow B^{\mathbb Z}_2(a)\rightarrow
B^{\mathbb Z}\rightarrow G_n^m.
\end{equation*}
Together with all imbedding and projections of  all mentioned groups
$G_n=B(n,{\mathbb R})$ we have:
\begin{equation*}\label{all-Group2}
B(n,{\mathbb R})\stackrel{i_{n}^{n+1}}{\rightarrow} B(n+1,{\mathbb
R})\stackrel{i_{n}^{\infty}}{\rightarrow} B_0^{\mathbb N}\rightarrow
B_2(a)\rightarrow B^{\mathbb N}\rightarrow B(n+1,{\mathbb
R})\stackrel{p^{n}_{n+1}}{\rightarrow}B(n,{\mathbb R}),
\end{equation*}
where the imbedding $i_{n}^{n+1}$ and the projections $p^{n}_{n+1}$
are defined as follows:
\begin{equation*}\label{df.i}
B(n,{\mathbb R})\ni x\mapsto i_{n}^{n+1}(x)=x+E_{n+1,n+1}\in
B(n+1,{\mathbb R}),
\end{equation*}
\begin{equation*}\label{df.p}
B(n+1,{\mathbb R})\ni x=x^{n+1}x_n\mapsto p^{n}_{n+1}(x)=x_n\in
B(n,{\mathbb R}),
\end{equation*}
\begin{equation*}
\text{where}\quad x^{n+1}=I+\sum_{k=1}^nx_{kn+1}E_{kn+1},\quad x_n=I+\sum_{1\leq
k<m\leq n}x_{km}E_{km}.
\end{equation*}
For groups $G^{m}_n\simeq B(2n,{\mathbb R})$ defined by (\ref{G^m_n=}) consider the
homomorphism $p^{s,m,n}_{n+1}:G^{m}_{n+1} \mapsto G^{m}_n$ defined
as follows (for simplicity we define $p^{s,m,n}_{n+1}$ for $m=0$)
\begin{equation*}\label{df.p-symm}
G^{0}_{n+1}\ni x=x^{n+1}_{\uparrow}x_n x^{n}_{\to}\mapsto
p^{s,0,n}_{n+1}(x)=x_n\in G^{0}_n,
\end{equation*}
where
$$x^{n+1}_{\uparrow}=I+\sum_{-n<k< n+1 } x_{k,n+1}E_{k,n+1},\quad
x^{n}_{\to}=I+\sum_{-n<k\leq n+1 } x_{-n,k}E_{-n,k}.
$$
\begin{rem}
\label{orbit-fin}  The embedding $\widehat{B(n,{\mathbb
R})}\mapsto \widehat{B(n+1,{\mathbb R})}$ (resp.
$\widehat{G^{m}_n}\mapsto \widehat{G^{m}_{n+1}}$) is induced by
the homomorphism (\ref{df.p}) $p^{n}_{n+1}:B(n+1,{\mathbb
R})\mapsto B(n,{\mathbb R})$ (resp. by the homomorphism
(\ref{df.p-symm}) $p^{s,m,n}_{n+1}:G^{m}_{n+1} \mapsto G^{m}_n$ ).
So for $m\in {\mathbb Z}$ we get $\bigcup_{n\in{\mathbb N}}
\widehat{G^{(m)}_n}\subset\widehat{B_0^{\mathbb Z}} $.
Similarly, we have $\cup_{n\in{\mathbb N}}\widehat{B(n,{\mathbb
N})}\subset\widehat{B_0^{\mathbb N}}$
\end{rem}
Let us denote by $B_2^{\mathbb N}(a)$ (resp. $B_2^{\mathbb Z}(a)$)
the completion of the subgroup $B_0^{\mathbb N}\subset {\rm
GL}_0(2\infty,{\mathbb R})$ (resp. $B_0^{\mathbb Z}\subset {\rm
GL}_0(2\infty,{\mathbb R})$) in the Hilbert-Lie group ${\rm
GL}_2(a)$. Since (see \cite{Kos88})
$$
B_0^{\mathbb N}=\bigcap_{a\in{\mathfrak A}}B_2^{\mathbb
N}(a)\quad\text{(resp.}\quad B_0^{\mathbb Z}=\bigcap_{a\in{\mathfrak
A}}B_2^{\mathbb Z}(a))
$$ we conclude that
$$
\widehat{B_0^{\mathbb N}}=\bigcup_{a\in{\mathfrak
A}}\widehat{B_2^{\mathbb N}(a)}\quad\text{(resp.}\quad
\widehat{B_0^{\mathbb Z}}=\bigcup_{a\in{\mathfrak
A}}\widehat{B_2^{\mathbb Z}(a)}).
$$
It leaves to describe $\widehat{B_2^{\mathbb N}(a)}$ (resp.
$\widehat{B_2^{\mathbb Z}(a)}$) for all $a\in{\mathfrak A}$. The
problem of developing the orbit method for the Hilbert-Lie group
$B_2^{\mathbb N}(a)$ (resp. $B_2^{\mathbb Z}(a)$) could be  easier,
since the corresponding Lie algebra ${\mathfrak b}_2^{\mathbb N}(a)$
(resp. ${\mathfrak b}_2^{\mathbb Z}(a)$) is a Hilbert-Lie algebra,
the dual $({\mathfrak b}_2^{\mathbb N}(a))^*$ (resp. $({\mathfrak
b}_2^{\mathbb Z}(a))^*$) and the pairing between ${\mathfrak
b}_2^{\mathbb N}(a)$ (resp. ${\mathfrak b}_2^{\mathbb Z}(a)$) and
$({\mathfrak b}_2^{\mathbb N}(a))^*$ (resp. $({\mathfrak b}_2^{\mathbb Z}(a))^*$) are well defined
(see  subsection \ref{orb-meth,1-step}).

Using (\ref{all-Group2}) we conclude
\begin{equation}
\label{ind,proj} B_0^{\mathbb N}= \varinjlim_{n,i}B(n,{\mathbb R}),\quad
B_0^{\mathbb N}=\varprojlim_{a}B^{\mathbb N}_2(a),\quad B^{\mathbb
N}=\varprojlim_{n,p}B(n,{\mathbb R}),
\end{equation}
\begin{equation*}
\widehat{B_0^{\mathbb N}}\supset \widehat{B^{\mathbb N}_2(a)}\supset
\widehat{B^{\mathbb N}},
\end{equation*}
finally we conclude that
\begin{equation}
\label{hat(fin,hilb,arb)}
 \widehat{B_0^{\mathbb N}}=\bigcup_{a\in {\mathfrak
A}}\widehat{B^{\mathbb N}_2(a)},\quad \widehat{B^{\mathbb
N}}=\bigcup_{n\in{\mathbb N}}\widehat{G_n}=\bigcup_{n\in{\mathbb
N}}\widehat{B(n,{\mathbb R})}.
\end{equation}
The similar relations holds also for  groups $B_0^{\mathbb Z}\subset B^{\mathbb Z}_2(a)\subset B^{\mathbb Z}$.
\begin{df} We call the representation of the group
$G=\varinjlim_{n}G_n$ {\rm local} if it
depends only on the elements of the subgroup $G_n$ for some fixed
$n\in {\mathbb N}$.
\end{df}
\index{representation!local}
The last relation in (\ref{ind,proj}) and (\ref{hat(fin,hilb,arb)}) we can reformulated as follows:
\begin{thm}(V.L.~Ostrovsky, PhD dissertation, 1986).
The class of all irreducible unitary {\rm local representations}  of the group
$B_0^{\mathbb N}=\varinjlim_{n}B(n,{\mathbb R})$ coincides  with the
class $\bigcup_{n}\hat{G_n}$.
\end{thm}
\section{Appendix 1. Gauss decompositions}
\label{Append-B-0^Z}
\subsection{Gauss decomposition of $n\times n$ matrices}
\label{Gauss-dec}
We need some decomposition of the matrix $C\in {\rm Mat}(n,\mathbb C)$. Let us denote  by
$$
M^{i_1i_2...i_r}_{j_1j_2...j_r}(C),\,\, 1\leq i_1<...<i_r\leq n,\,\, 1\leq j_1<...<j_r\leq n
$$
the  minors of the matrix $C$ with
$i_1,i_2,...,i_r$ rows and $j_1,j_2,...,j_r$ columns.
%
\begin{thm}[Gauss decomposition, \cite{Gan58}]
\label{t.C=UDL}
 A matrix $C\in {\rm Mat}(n,\mathbb C)$  admits the following decomposition
$C=LDU$ (Gauss decomposition),
\begin{equation}
\label{C=UDL}
\left(\begin{array}{cccc}
c_{11}&c_{12}&\dots&c_{1n}\\
c_{21}&c_{22}&\dots&c_{2n}\\
            &&\dots&\\
c_{n1}&c_{n2}&\dots&c_{nn}\\
\end{array}\right)=
\left(\begin{array}{cccc}
1&0&\dots&0\\
l_{21}&1&\dots&0\\
            &&\dots&\\
l_{n1}&l_{n2}&\dots&1\\
\end{array}\right)
\left(\begin{array}{cccc}
d_{1}&0&\dots&0\\
0&d_{2}&\dots&0\\
            &&\dots&\\
0&0&\dots&d_{n}\\
\end{array}\right)
\left(\begin{array}{cccc}
1&u_{12}&\dots&u_{1n}\\
0&1&\dots&u_{2n}\\
            &&\dots&\\
0&0&\dots&1\\
\end{array}\right)
\end{equation}
where $L$ (resp. $U$) is lower (resp. upper) triangular matrix and $D$ a diagonal matrix if and only if all principal
minors of the matrix $C$ are different from zeros i.e.
$M^{1,2,\dots, k}_{1,2,\dots,k}(C)\not=0,\,\,1\leq k\leq n$. Moreover the matrix elements of the matrices $L,\,\,U$ and $D$ are
given by the formulas (see \cite[Ch.II, \S 4, (44), (45)]{Gan58})
\begin{equation}
\label{U,L=}
l_{mk}=\frac{M^{1,2,\dots, k-1,m}_{1,2,\dots, k-1,k}(C)}{M^{1,2,\dots, k-1,k}_{1,2,\dots, k-1,k}(C)},\quad
u_{km}=\frac{M^{1,2,\dots, k-1,k}_{1,2,\dots, k-1,m}(C)}{M^{1,2,\dots, k-1,k}_{1,2,\dots, k-1,k}(C)},1\leq k<m\leq n,
\end{equation}
\begin{equation}
\label{D=}
d_1=M^1_1(C),\quad d_k=\frac{M^{1,2,\dots, k}_{1,2,\dots, k}(C)}{M^{1,2,\dots, k-1}_{1,2,\dots, k-1}(C)},\quad 2\leq k\leq n.
\end{equation}
\end{thm}
\begin{proof}
If we write $L^{-1}C=DU$, we get
$$
M^{1,2,\dots, k-1,k}_{1,2,\dots, k-1,k}(C)=
M^{1,2,\dots, k-1,k}_{1,2,\dots, k-1,k}(L^{-1}C)=M^{1,2,\dots, k-1,k}_{1,2,\dots, k-1,k}(DU)=d_1\dots d_k,
$$
this implies (\ref{D=}). Moreover, we get also
$$
M^{1,2,\dots, k-1,k}_{1,2,\dots, k-1,m}(L^{-1}C)=M^{1,2,\dots, k-1,k}_{1,2,\dots, k-1,m}(C)=M^{1,2,\dots, k-1,k}_{1,2,\dots, k-1,m}(DU)=
d_1\dots d_ku_{km},\,\,k<m,
$$
this implies the second formula in (\ref{U,L=}). Similarly if we write $CU^{-1}=LD$ we get
$$
M^{1,2,\dots, k-1,m}_{1,2,\dots, k-1,k}(CU^{-1})=M^{1,2,\dots, k-1,m}_{1,2,\dots, k-1,k}(C)=M^{1,2,\dots, k-1,m}_{1,2,\dots, k-1,k}(LD)=
d_1\dots d_kl_{mk},\,\,k<m,
$$
this implies the first formula in (\ref{U,L=}).
\end{proof}
%

\subsection{Gauss  decomposition of infinite order matrices}
\label{Gauss-dec-inf}
Let us consider the infinite  matrix $C,L,D,U\in {\rm Mat}(\infty,\mathbb C)$.
\begin{thm}[Gauss decomposition $C=LDU$]
\label{t.C=LDU,Z1}
 A matrix $C\in {\rm Mat}(\infty,\mathbb C)$  admits the following decomposition
$C=LDU$ (Gauss decomposition),
\begin{equation}
\label{C=LDU,inf}
\left(\begin{smallmatrix}
c_{11}&c_{12}&\dots&c_{1n}&\dots\\
c_{21}&c_{22}&\dots&c_{2n}&\dots\\
            &&\dots&&\dots\\
c_{n1}&c_{n2}&\dots&c_{nn}&\dots\\
 &&\dots&&\dots\\
 \end{smallmatrix}\right)=
\left(\begin{smallmatrix}
1&0&\dots&0\\
l_{21}&1&\dots&0&\dots\\
            &&\dots&&\dots\\
l_{n1}&l_{n2}&\dots&1&\dots\\
 &&\dots&&\dots\\
 \end{smallmatrix}\right)
\left(\begin{smallmatrix}
d_{1}&0&\dots&0&\dots\\
0&d_{2}&\dots&0&\dots\\
            &&\dots&&\dots\\
0&0&\dots&d_{n}&\dots\\
 &&\dots&&\dots\\
\end{smallmatrix}\right)
\left(\begin{smallmatrix}
1&u_{12}&\dots&u_{1n}&\dots\\
0&1&\dots&u_{2n}&\dots\\
            &&\dots&&\dots\\
0&0&\dots&1&\dots\\
 &&\dots&&\dots\\
 \end{smallmatrix}\right)
 \end{equation}
where $L$ (resp. $U$) is lower (resp. upper) triangular matrix and $D$ a diagonal matrix of infinite order if and only if all principal
minors of the matrix $C$ are different from zeros i.e.
$M^{1,2,\dots, k}_{1,2,\dots,k}(C)\not=0,\,\,k\in{\mathbb N}$. Moreover the matrix elements of the matrices $L,\,\,U$ and $D$ are
given by the same formulas as in the Theorem \ref{t.C=UDL}:
\begin{equation}
\label{U,L=inf}
l_{mk}=\frac{M^{1,2,\dots, k-1,m}_{1,2,\dots, k-1,k}(C)}{M^{1,2,\dots, k-1,k}_{1,2,\dots, k-1,k}(C)},\quad
u_{km}=\frac{M^{1,2,\dots, k-1,k}_{1,2,\dots, k-1,m}(C)}{M^{1,2,\dots, k-1,k}_{1,2,\dots, k-1,k}(C)},\quad k,m \in {\mathbb N},\,\, k<m,
\end{equation}
\begin{equation}
\label{D=,inf}
d_1=M^1_1(C),\quad d_k=\frac{M^{1,2,\dots, k}_{1,2,\dots, k}(C)}{M^{1,2,\dots, k-1}_{1,2,\dots, k-1}(C)},\quad k\in {\mathbb N},\,\,k>1.
\end{equation}
\end{thm}
\begin{proof} The proof repeat word by word the proof of the Theorem~\ref{t.C=UDL}.\end{proof}
\section{Appendix 2. One elementary fact  concerning abelian von Neumann algebras}
\label{Append-B-0^Z-Neumann}
Let $(X,{\mathcal F},\mu)$ be a {\it measurable space}, with a finite measure $\mu(X)<\infty$, where ${\mathcal F}$
is a sigma-algebra.
Consider the set $(f_n)=(f_n)_{n\in \mathbb N}$ of measurable real valued functions on $X$ i.e. $f_n:X\mapsto {\mathbb R}$.
Denote by $B(H)$ the von Neumann algebra of all bounded operators in the Hilbert space $H=L^2(X,\mu)$ and let ${\mathfrak A}^{(f_n)} (\in B(H))$
be a von Neumann algebra generated by operators $U_n(t)$ of multiplication by functions $\exp(itf_n(x)),\,\,n\in {\mathbb N}$
$$
{\mathfrak A}^{(f_n)}=\big(U_n(t)=e^{itf_n}\mid n\in {\mathbb N},\,t\in {\mathbb R}\big)''.
$$
We are interesting in the following {\it question}. Let $f_n\rightarrow f$ as $n\rightarrow\infty$ in some sense.
When $U(t)=e^{itf}\in {\mathfrak A}^{(f_n)}$ for all $t\in {\mathbb R} $?

Since ${\mathfrak A}^{(f_n)}$ is a von Neumann algebra it is sufficient to find when the strong convergence of the unitary operators
in the space $H$ holds  i.e. $s.\lim_nU_n(t)=U(t)$, where the operators  $U_n(t),\,\,n\in{\mathbb N}$ and  $U(t)$ are defined as follows
$$
(U_n(t)g)(x)=e^{itf_n(x)}g(x),\quad (U(t)g)(x)=e^{itf(x)}g(x),\quad g\in L^2(X,\mu),\,\,t\in {\mathbb R}.
$$
\begin{lem}
\label{lim-f(n)=f-in-von-Neu}
Let $f_n\rightarrow f$ as $n\rightarrow\infty$ pointwise almost everywhere, then $s.\lim_nU_n(t)=U(t)$ hence $U(t)=e^{itf}\in {\mathfrak A}^{(f_n)}$.
\end{lem}
\begin{proof} For $g\in H$ we get
$$
\Vert (U_n(t)-U(t))g\Vert^2=\int_X \mid \big(e^{itf_n(x)}-e^{itf(x)}\big)g(x)\mid^2d\mu(x)=
$$
$$
\int_X \mid e^{itf_n(x)-itf(x)}-1\mid^2\mid g(x)\mid^2d\mu(x)=\int_X \mid e^{it\alpha_n(x)}-1\mid^2\mid g(x)\mid^2d\mu(x)
\rightarrow 0
$$
as $n\rightarrow\infty$, if $\alpha_n(x):=f_n(x)-f(x)\rightarrow 0$ pointwise almost everywhere by  Lebesgue's dominated convergence theorem.
\index{theorem!Lebesgue's!dominated convergence}
\end{proof}

\end{document}